\documentclass{amsart}

\usepackage{amssymb}
\usepackage{tikz}
\usepackage{verbatim}
\usepackage{hyperref}
\usepackage[english]{babel}
\usepackage{wasysym}
\usepackage[shortlabels]{enumitem}

\usepackage[colorinlistoftodos, textsize=small]{todonotes}
\usepackage{color, soul}
\usepackage{mathtools}
\usepackage{mathrsfs}

\usepackage{thmtools}
\usepackage{thm-restate}

\newtheorem{theorem}{Theorem}[section]
\newtheorem{lemma}[theorem]{Lemma}
\newtheorem{corollary}[theorem]{Corollary}

\theoremstyle{definition}
\newtheorem{definition}[theorem]{Definition}
\newtheorem{notation}[theorem]{Notation}
\newtheorem{example}[theorem]{Example}

\newtheorem{question}[theorem]{Question}
\newtheorem{proposition}[theorem]{Proposition}

\newtheorem{conjecture}[theorem]{Conjecture}

\theoremstyle{remark}
\newtheorem{remark}[theorem]{Remark}
\theoremstyle{remark}

\numberwithin{equation}{section}

\newcommand{\sym}{\operatorname{Sym}}
\newcommand{\pol}{\operatorname{Pol}}
\newcommand{\proj}{\mathscr{P}}

\newcommand{\id}{\operatorname{id}}

\newcommand{\csp}{\operatorname{CSP}}
\newcommand{\pcsp}{\operatorname{PCSP}}
\newcommand{\yes}{\bf{YES}}
\newcommand{\no}{\bf{NO}}
\newcommand{\dom}{\operatorname{Dom}}

\newcommand{\ar}{\operatorname{ar}}

\newcommand{\zf}{\operatorname{ZF}}
\newcommand{\zfc}{\operatorname{ZFC}}
\newcommand{\ac}{\operatorname{AC}}
\newcommand{\acf}{\ac^*(<\!\omega)}
\newcommand{\ufl}{\operatorname{UL}}
\newcommand{\kw}{\operatorname{KW}}
\newcommand{\hh}{H}
\newcommand{\nae}{\operatorname{NAE}}
\newcommand{\val}{\operatorname{val}}
\newcommand{\ex}{\operatorname{ex}}
\newcommand{\power}{\mathcal{P}}
\newcommand{\upower}{\mathcal{U}}
\newcommand{\free}{\mathbf{F}}

\newcommand{\p}{\mathbf{P}}
\newcommand{\np}{\mathbf{NP}}
\newcommand{\npc}{$\mathbf{NP}$-complete}
\newcommand{\nph}{$\mathbf{NP}$-hard}

\newcommand{\fa}{\mathfrak{A}}
\newcommand{\fb}{\mathfrak{B}}
\newcommand{\fc}{\mathfrak{C}}
\newcommand{\fd}{\mathfrak{D}}

\newcommand{\ff}{\mathfrak{F}}
\newcommand{\fg}{\mathfrak{G}}
\newcommand{\fh}{\mathfrak{H}}
\newcommand{\fm}{\mathfrak{M}}

\newcommand{\ii}{\mathfrak{I}}

\newcommand{\fs}{\mathfrak{S}}

\begin{document}

\title{Equivalences of promise compactness principles}
\author{Bertalan Bodor}
\address{HUN-REN Alfréd Rényi Institute of Mathematics, Reáltanoda utca 13-15, H-1053, Budapest}
\email{bodor@renyi.hu}
\thanks{The author has been funded by the European Research Council (Project POCOCOP, ERC Synergy Grant 101071674). Views and opinions expressed are however those of the authors only and do not necessarily reflect those of the European Union or the European Research Council Executive Agency. Neither the European Union nor the granting authority can be held responsible for them.}
\maketitle

\begin{abstract}
	For a pair of finite relational structures $(\fa,\fb)$ such that $\fa$ homomorphically maps to $\fb$ we denote by $K_{(\fa,\fb)}$ the following statement: \emph{for all structures $\mathfrak{I}$ with the same signature as $\fa$ if all finite substructures of $\mathfrak{I}$ homomorphically maps to $\fa$ then $\mathfrak{I}$ homomorphically maps to $\fb$}. 
	
	In this article, we show that if $(\fa,\fb)$ has no Ol\v{s}\'{a}k polymorphism, then $K_{(\fa,\fb)}$ is equivalent to the ultrafilter principle over $\zf$. This includes the statements $K_{(K_3,K_5)}$ and $K_{(H_2,H_c)}$ for all $c\geq 2$ where $K_n$ denotes the clique of size $n$ and $H_k$ denotes the ternary not-all-equal structure on a $k$-element set. This means, for example, that in any $\zf$ model, if every finitely 3-colourable graph can be coloured by 5 colours then all these graphs can in fact be coloured by 3 colours.


\end{abstract}

\section{Introduction}\label{sect:intro}

	For a finite relational structure $\fa$ we write $K_{\fa}$ for the following statement.
	
\begin{itemize}
\item \emph{For every relation structure $\ii$ with the same signature as $\fa$ if all finite\\ substructures of $\ii$ homomorphically map to $\fa$ then so does $\ii$.}
\end{itemize}
	
	We refer to the statement $K_{\fa}$ as the \emph{compactness principle} over $\fa$. Note that all statements $K_{\fa}$ follow from the compactness theorem of propositional logic, and thus they are all theorems of $\zfc$. In fact, we know that the compactness theorem (even for first-order logic), and thus all compactness principles $K_{\fa}$, follow already from the ultrafilter lemma which is strictly weaker than the axiom of choice~\cite{jech1977axiom} over $\zf$. However, different compactness principles have different strengths in choiceless set theory. For instance, by a theorem of L\"{a}uchli~\cite{lauchli1971coloring} we know that $K_{K_3}$ is in fact equivalent to the ultrafilter lemma over $\zf$ meaning that $K_{K_3}$ is (one of) the strongest compactness principles. On the other hand, it is fairly easy to see that $K_{K_2}$ is equivalent to $\ac(2)$, i.e., that axiom of choice for families of 2-element sets which is much weaker than the ultrafilter lemma~\cite{jech1977axiom}. 
	
\subsection{CSPs and compactness principles} For a relational structure $\fa$ the CSP (constraint satisfaction problem) over $\fa$, denoted by $\csp(\fa)$, is the computation problem where the input is any structure $\ii$ with the same signature of $\fa$, and we need to decide whether $\ii$ homomorphically maps to $\fa$. 
As has been observed at many places in the literature, the strength of $K_{\fa}$ is closely related the computation complexity of the CSP over $\fa$; in general the harder the CSP is over $\fa$ the stronger the statement $K_{\fa}$ seems to be~\cite{rorabaugh2017logical,katay2023csp,tardif2025constraint}. For example $\csp(K_3)$, i.e., 3-colourability of graphs, is {\npc} whereas $K_{K_2}$, i.e., 2-colourability is polynomial time decidable. This observation can be formalised using the notion of \emph{primitive positive (or pp-) constructions}: we say that a structure $\fa$ \emph{pp-constructs} a structure $\fb$ if $\fb$ is homomorphically equivalent to some pp-power of $\fa$~\cite{wonderland}. 
One of the most important features of pp-constructions is that they give reductions between complexity of CSPs. More specifically, if $\fa$ pp-constructs $\fb$ then $\csp(\fb)$ polynomial time (in fact LOGSPACE) reduces to $\csp(\fa)$~\cite{wonderland}. Following the terminology of~\cite{brunar2025janus} we say that finite structure $\fa$ is \emph{omniexpressive} if it pp-constructs all finite structures. We know, for instance, that $K_3$ is omniexpressive, and thus omniexpressivity is also equivalent to the pp-constructability of $K_3$. In particular, if $\fa$ pp-constructs $K_3$ then $\csp(\fa)$ is {\npc}. On the other hand, it was shown independently by Bulatov~\cite{BulatovFVConjecture} and Zhuk~\cite{zhuk2020proof} that all CSPs over not omniexpressive finite structures are solvable in polynomial time, thereby confirming the famous dichotomy conjecture by Feder and Vardi~\cite{FederVardi}.

	As for compactness principles some analogous results were obtained in a recent work of Kátay, Tóth, and Vidnyánszky~\cite{katay2023csp}. Here, it has been shown that if $\fa$ pp-constructs $\fb$ then $K_{\fa}$ implies $K_{\fb}$. This implies immediately that all compactness principles over finite omniexpressive structures are equivalent to each other, and thus by the theorem of L\"{a}uchli, they are also equivalent to the ultrafilter lemma. On the other hand, we know if $\fa$ is finite and not omniexpressive then $K_{\fa}$ is strictly weaker than $K_{K_3}$. In fact, the following result is shown in the aforementioned paper.
	
\begin{theorem}[\cite{katay2023csp}, Theorem 2.18]\label{thm:ktv}
	There exists a model of $\zf$ where for a finite structure $\fa$ the compactness principle $K_{\fa}$ holds if and only if $\fa$ is not omniexpressive.
\end{theorem}

	Note that combining the finite-domain CSP dichotomy theorem with Theorem~\ref{thm:ktv} we obtain that $\csp(\fa)$ is {\nph} if and only if $K_{\fa}$ is equivalent to the ultrafilter lemma, unless $\p=\np$.
\subsection{Promise CSPs}

	We call a pair of structures $(\fa,\fb)$ a \emph{promise template} if $\fa$ and $\fb$ have the same signature and there is a homomorphism from $\fa$ to $\fb$. A promise template $(\fa,\fb)$ is called finite if both $\fa$ and $\fb$ are finite.
	The \emph{promise CSP}, or \emph{PCSP}, over a finite promise template $(\fa,\fb)$, denoted by $\pcsp(\fa,\fb)$, is the decision problem where for an input $\ii$ with the same signature as $\fa$ (or $\fb)$ we need to output
\begin{itemize}
\item YES if $\ii$ homomorphically maps to $\fa$, and
\item NO if $\ii$ does not homomorphically map to $\fb$.
\end{itemize}

	The ``promise'' in this setting is that we are given an input where one of the two cases above holds. Note that in this notation $\csp(\fa)$ is the same as $\pcsp(\fa,\fa)$. The notion of promise PCSPs was introduced in~\cite{austrin20172+varepsilon}, although some variations, such as approximate graph colourings (i.e., PCSPs over $(K_k,K_c)$) already appeared earlier in the literature~\cite{garey1976complexity,dinur2005hardness,guruswami2004hardness,huang2013improved,brakensiek2016new}. As it turns out, pp-constructions, and thus also omniexpressivity, can be generalised to promise templates, and they still imply complexity reductions in the same way~\cite{barto2021algebraic}. 
	For example, it follows from the results of~\cite{brakensiek2016new} 
	that $(K_k,K_{2k-2})$ is omniexpressive for all $k\geq 3$, and therefore $\pcsp(K_k,K_{2k-2})$ is {\npc}.
	However, as opposed to finite-domain CSPs, there are many non-omniexpressive finite promise templates which are known to have hard PCSPs. Some notable examples are
	
\begin{itemize}
\item $(K_k,K_{2k-1})$ for $k\geq 3$~\cite{barto2021algebraic},
\item $(K_k,K_c)$ where $k\geq 4$ and $c={k\choose \lfloor k/2 \rfloor}-1$~\cite{krokhin2023topology},
\item $(C_{2k+1},K_4)$ for $k\geq 1$ where $C_{\ell}$ denotes the (undirected) cycle of length $\ell$~\cite{avvakumov2025hardness}, and
\item $(H_2,H_c)$ for $c\leq 2$ where $H_k$ denotes the structures on a $k$-element set with a single ternary not-all-equal relation~\cite{dinur2005hardness}.
\end{itemize}

	Further hardness results have been shown in~\cite{austrin20172+varepsilon,filakovsky2023hardness}. We remark that all PCSPs over $(K_k,K_c)$ where $3\leq k\leq c$ are in fact conjectured to be hard but already the complexity of $\pcsp(K_3,K_6)$ is open~\cite{barto2021algebraic}.
	
\subsection{Promise compactness}

	Following the idea of promise CSPs, we introduce the promise version of compactness principles in the following way. Let $(\fa,\fb)$ be a promise template. We then write $K_{(\fa,\fb)}$ for the following statement, to which we refer as the \emph{promise compactness principle} over $(\fa,\fb)$.

\begin{itemize}
\item \emph{For every relation structure $\ii$ with the same signature as $\fa$ if all finite\\ substructures of $\ii$ homomorphically map to $\fa$ then $\ii$ homomorphically maps to $\fb$.}
\end{itemize}

	It follows from a straightforward generalization of the proofs of~\cite{katay2023csp} that if $(\fa,\fb)$ pp-constructs $(\fc,\fd)$ then $K_{(\fa,\fb)}$ implies $K_{(\fc,\fd)}$ (over $\zf$) (see Section~\ref{sect:promise}). Thus, it still holds that if $(\fa,\fb)$ is omniexpressive then $K_{(\fa,\fb)}$ is equivalent to the ultrafilter lemma. By the results of~\cite{brakensiek2016new} mentioned above, this includes all promise templates of the form $(K_k,K_c)$ where $3\leq k\leq c\leq 2k-2$. This means, for example, that if in a $\zf$ model every finitely 3-colourable graph can be coloured by 4 colours then all these graphs can in fact be coloured by 3 colours. So we can get rid of one of the colours without any choice axiom. 

	
	Since in the promise setting omniexpressivity does not correspond to hardness of PCSPs this motivates the question as to which promise compactness principles are equivalent to $K_{K_3}$, or more concretely: is it true that if $\pcsp(\fa,\fb)$ is hard, then $K_{(\fa,\fb)}$ is equivalent to the ultrafilter lemma? In this article, we settle this question in the following cases.
	
\begin{theorem}\label{thm:int1}
	Any of the following statements are equivalent to the ultrafilter lemma over $\zf$.
\begin{enumerate}
\item $K_{(K_k,K_{2k-1})}$ with $k\geq 3$.
\item $K_{(H_k,H_c)}$ with $2\leq k\leq c$.
\end{enumerate}
\end{theorem}

	As for finitely 3-colourable graphs this means that we can also get rid of a second colour without any choice axiom.

\subsection{Polymorphisms}

	Let $(\fa,\fb)$ be a promise template. Then a map $f:A^k\rightarrow B$ is called a \emph{polymorphism} of $(\fa,\fb)$ if it is a homomorphisms from the categorical power $\fa^k$ to $\fb$. A polymorphism of a single structure $\fa$ is a polymorphism of $(\fa,\fa)$. Polymorphisms of a promise template always form a \emph{minion}, i.e., a minor-closed set of operations, 
	and polymorphisms of a single structure from a clone. We write $\pol(\fa,\fb)$ for the polymorphism minion of $(\fa,\fb)$, and we write $\pol(\fa,\fa)\coloneqq \pol(\fa)$. Polymorphisms play an essential role in the study of CSPs and PCSPs. For example, we know by the results of~\cite{barto2021algebraic} that for finite $\fa,\fb,\fc,\fd$ the promise template $(\fa,\fb)$ pp-constructs $(\fc,\fd)$ if and only if there exists a \emph{minion homomorphism} from $\pol(\fa,\fb)$ to $\pol(\fc,\fd)$. By the results of~\cite{Cyclic} we know the following algebraic dichotomy for finite structures.
	
\begin{theorem}\label{thm:cyclic}
	Let $\fa$ be a finite structure. Then $\fa$ is not omniexpressive if and only if $\fa$ has a cyclic polymorphism, i.e., an $h\in \pol(\fa)$ satisfying the identity $h(x_0,\dots,x_{k-1})=h(x_1,\dots,x_{k-1},x_0)$ for some $k\geq 2$.
\end{theorem}

	We remark that Theorem~\ref{thm:cyclic}, in fact, plays a crucial role in both Zhuk's and Bulatov's proof of the finite-domain CSP dichotomy theorem, as well as in the proof of Theorem~\ref{thm:ktv}, although in the former two proofs a weaker version of Theorem~\ref{thm:cyclic} (with \emph{weak near-unanimity} polymorphisms) is sufficient.
	
	In a recent work of Barto and Kozik~\cite{barto2022combinatorial} a new reduction between promise CSPs was introduced which is strictly more powerful than pp-constructions. This reduction involves a generalization of minion homomorphisms which we will simply call \emph{weak minion homomorphisms} (originally called \emph{$(d,r)$-minion homomorphisms}, see Definition~\ref{def:min}). This reduction together with the results of~\cite{nakajima2025complexity} gives a new proof of the hardness of PCSPs over $(H_k,H_c): 2\leq k\leq c$ and $(K_k,K_{2k-1}): k\geq 3$. Note however that this reduction is purely algebraic in nature, and as of now there is no clear description on the level of structures as to when this reduction is possible. Nevertheless, in this article, we show that this reduction is robust enough that it can also be used to infer implication between the corresponding promise compactness principles provided some weak version of $\ac$.
	
\begin{theorem}\label{thm:int2}
	Assume that for all $k\in \omega$ the axiom of choice holds for all families of $k$-element sets, and let us assume that there exists a weak minion homomorphism from $\pol(\fa,\fb)$ to $\pol(\fc,\fd)$. Then if $K_{(\fa,\fb)}$ holds, then so does $K_{(\fc,\fd)}$.
\end{theorem}
	
	Even though Theorem~\ref{thm:int2} uses a choice assumption in its formulation, we show that in the concrete cases where we need to apply it, this assumption automatically holds. More specifically, we show the following.

\begin{theorem}\label{thm:int3}
	Let us assume that $(\fa,\fb)$ is finite promise template which has no cyclic polymorphism and $K_{(\fa,\fb)}$ holds. Then for all $k\in \omega$ the axiom of choice holds for all families of $k$-element sets.
\end{theorem}
	
	An operation $f:A^6\rightarrow B$ is called \emph{Ol\v{s}ák} if it satisfies the identities $$f(x,x,y,y,y,x)=f(x,y,x,y,x,y)=f(y,x,x,x,y,y).$$
Using Theorems~\ref{thm:int2} and Theorem~\ref{thm:int3} we will show the following.

\begin{theorem}\label{thm:int4}
	Let us assume that $(\fa,\fb)$ is finite promise template which has no Ol\v{s}ák polymorphism. Then $K_{(\fa,\fb)}$ is equivalent to the ultrafilter lemma.
\end{theorem}

\subsection*{Outline of the paper}

	The paper is structured as follows. In Section~\ref{sect:prelims} we recall all notions and results in the literature that are relevant for this article. In Section~\ref{sect:promise} we introduce the notion of promise compactness and prove some results that follow easily from what is already available in the literature. In Section~\ref{sect:ac_fin} we discuss the relation of promise compactness and certain choice axioms for finite sets, and we prove Theorem~\ref{thm:int3}. In Section~\ref{sect:weak} we introduce weak minion homomorphisms and we show Theorem~\ref{thm:int2} by going through the proofs of the main results of~\cite{barto2022combinatorial} with a slight modification. This is the most technical part of the paper. Here, we also discuss how Theorems~\ref{thm:int1} and~\ref{thm:int4} can be derived from the two previously mentioned theorems. Finally, in Section~\ref{sect:open} we discuss some open problems and possible future research questions about promise compactness.

\section{Preliminaries}\label{sect:prelims}

	We denote by $\omega$ the set of natural numbers $\{0,1,\dots\}$, and we use the usual set theoretical convention that $n=\{0,1,\dots,n-1\}$ for $n\in \omega$. For any set $A$ and $B$ we use the notation $\prescript{B\!}{}{A}$ for the set of functions from $B$ to $A$. We think of $n$-tuples from a set $A$ as functions from $n$ to $A$, and if it does not cause confusion, we write $A^n$ (as opposed to $\prescript{n\!}{}{A}$) for the set of $n$-tuples from $A$. For a function $\prescript{B\!}{}{A}$ and $C\subseteq B$, we write $f|_C$ for the restriction of $f$ to $C$, i.e., $f|_C=f\cap (C\times A)$. For a subset $\mathcal{S}$ of $\prescript{B\!}{}{A}$ we write $\mathcal{S}|_C\coloneqq \{f|_C: f\in \mathcal{S}\}$. For a set $A$ we write $\power(A)$ for the set of all subsets of $A$, and for an $n\in \omega$ we write ${A\choose n}$ for the set of $n$-element subsets of $A$.
	
	By a \emph{relational signature} we mean a set of relational symbols $\mathcal{R}$ together with an arity function $\ar: \mathcal{R}\rightarrow \omega$. In this paper all signatures will be assumed to be relational and finite. A \emph{relational structure} is a triple $\fa=(A,\sigma,i)$ where $A$ is a set, called the \emph{domain set} of $\fa$, $\sigma$ is a relational signature, called the \emph{signature of $\fa$}, and $i$ maps each $R\in \mathcal{R}$ to a subset of $A^{\ar(R)}$. We will write $R^{\fa}$ for $i(R)$, and $\dom(\fa)$ for the domain set of $\fa$. We say that two structures $\fa$ and $\fb$ are \emph{similar} if they have the same signature. We use the notational convention that structures are denoted by Fraktur letters
, and their domain sets are denoted by the corresponding Latin letters. When talking about finite structures, we usually assume that their domain is $n$ for some $n\in \omega$.
	
	For a structure $\fa$ and a subset $B\subseteq A=\dom(\fa)$, we write $\fa|_B$ for the \emph{substructure of $\fa$ induced on $B$}, i.e., $\fa|_B$ is the structure whose domain set is $B$, is similar to $\fa$, and for all relational symbols $R$ in their signature we have $R^{\fa|_B}=R^{\fa}\cap B^{\ar(R)}$.
	
	If $\fa$ and $\fb$ are similar structures then a map $h: A\rightarrow B$ is called a \emph{homomorphism} from $\fa$ to $\fb$ if for all relational symbols $R$ in their signature whenever $(a_1,\dots,a_{\ar(R)})\in R^{\fa}$ then $(h(a_1),\dots,h(a_{\ar(R)}))\in R^{\fb}$. 
	
	For a $k\in \omega$ and a relational structure $\fa$ the \emph{$k$-th power of $\fa$}, denoted by $\fa^k$, is defined to be the structure whose domain is $A^k$, similar to $\fa$, and for any relational symbol $R$ with $\ar(R)=n$ we have $((a_{11},\dots,a_{1k}),\dots,((a_{n1},\dots,a_{nk}))\in R^{\fa^k}$ if and only if $(a_{1i},\dots,a_{ni})\in R^{\fa}$ for all $i=1,\dots,k$.

	We denote by $K_n$ the complete graph with vertex set $n$, and we write $H_n$ for the structure on $n$ with a single ternary \emph{not-all-equal} relation $\nae$, i.e., $$\nae\coloneqq\{\prescript{3}{}{n}\setminus \{(i,i,i): i\in n\}\}.$$

\subsection{Primitive positive constructions and minion homomorphism}

	In this subsection we are giving a quick introduction to pp-constructions of promise templates and recall some important results that are necessary for the proofs of this paper. Most of the material in this subsection are taken from~\cite{barto2021algebraic}.
	
	We assume that the reader is familiar with the usual notions in first-order logic.
		
\begin{definition}
	We say that a first-order formula $\varphi$ is \emph{primitive positive}, or \emph{pp} for short, if $\varphi$ can be built up from atomic formulas by only using conjunctions and atomic formulas.
\end{definition}
	
\begin{definition}\label{ppp}
	We write $\fa\rightarrow \fb$ if there exists a homomorphism from $\fa$ to $\fb$.\\
	A \emph{promise template} is a pair $(\fa,\fb)$ of similar structures with $\fa\rightarrow \fb$. A promise template $(\fa,\fb)$ is called finite is both $\fa$ and $\fb$ are finite. In this paper all promise templates will be finite.

	The \emph{PCSP} (P=promise) over $(\fa,\fb)$, in notation $\pcsp(\fa,\fb)$, is the decision problem where the input is a finite structure similar to $\fa$ and we need to output {\yes} if $\ii\rightarrow \fa$ and {\no} if $\ii\nrightarrow \fb$. The CSP over $\fa$, in notation $\csp(\fa)$, is $\pcsp(\fa,\fa)$.

	Let $(\fa,\fb)$ and $(\fc,\fd)$ be promise templates. Then
\begin{enumerate}[(i)]
\item $(\fc,\fd)$ is called a \emph{homomorphic relaxation} of $(\fa,\fb)$ if all structures $\fa,\fb,\fc,\fd$ are similar, $\fc\rightarrow \fa$ and $\fb\rightarrow \fd$.
\item $(\fc,\fd)$ is called an \emph{($n$th) pp-power} of $(\fa,\fb)$ if $C=A^n, D=B^n$, and for every relational symbol $R$ of arity $k$ in the signature of $\fc$ there exists a pp-formula $\varphi_R$ in the signature of $\fa$ with free variables $x_i^j: 1\leq i\leq k, 1\leq j\leq n$ such that
\begin{itemize}
\item for all $a_i^j\in A: 1\leq i\leq k, 1\leq j\leq n$ we have $((a_1^1,\dots,a_1^n),\dots,(a_k^1,\dots,a_k^n))\in R^{\fc}$ iff $\varphi_R(a_1^1,\dots,a_k^n)$ holds in $\fa$, and
\item for all $b_i^j\in B: 1\leq i\leq k, 1\leq j\leq n$ we have $((b_1^1,\dots,b_1^n),\dots,(b_k^1,\dots,b_k^n))\in R^{\fd}$ iff $\varphi_R(b_1^1,\dots,b_k^n)$ holds in $\fb$.
\end{itemize}
\end{enumerate}

	We say that $(\fc,\fd)$ is \emph{pp-constructible} from $(\fa,\fb)$, or $(\fa,\fb)$ \emph{pp-constructs} $(\fc,\fd)$, if $(\fc,\fd)$ is a homomorphic relaxation of a pp-power of $(\fa,\fb)$.
\end{definition}

	
	Note that for templates of the form $(\fa,\fa)$ homomorphic relaxation collapses to homomorphic equivalence and pp-power is just the usual pp-power for single structures.

\begin{definition}
	We denote by $\pi_i^k$ the \emph{$i$th projection map} on $k$-tuples, i.e., the map which assigns to any $k$-tuple its $i$th coordinate.\\
	Let us fix some sets $A$ and $B$. Then a set of functions $\mathcal{M}$ from some finite power $A^k: k\in \omega\setminus \{0\}$ to $B$ is called a \emph{(function) minion} if it is closed under taking minors, i.e., for any $f\in \mathcal{M}, f: A^k\rightarrow B$ and $\pi: k\rightarrow \ell$ we have $f_{\pi}\in \mathcal{M}$ where $$f_{\pi}: (x_0,\dots,x_{\ell-1})\mapsto f(x_{\pi(0)},\dots,x_{\pi(k-1)}).$$ For $f,g\in \mathcal{M}$ and for $\pi: k\rightarrow \ell$ we sometimes write $f\xrightarrow{\pi}g$ to indicate $g=f_{\pi}$.
	
	We say that a minion $\mathcal{M}$ has \emph{a finite domain} if $A$ and $B$ can be chosen to be finite.
	
	If $\mathcal{M}$ and $\mathcal{N}$ are minions, then a map $\xi: \mathcal{M}\rightarrow \mathcal{N}$ is called a \emph{minion homomorphism} if $\xi$ preserves the arities of functions, and for all $k$-ary $f\in \mathcal{M}$ and $\pi: k\rightarrow \ell$ we have $\xi(f_{\pi})=(\xi(f))_{\pi}$.
	
	We say that a map $f:A^k\rightarrow B$ is a \emph{polymorphism} of a PCSP template $(\fa,\fb)$ if $f$ is a homomorphism from $\fa^k$ to $\fb$. The set of polymorphisms of $(\fa,\fb)$ is denoted by $\pol(\fa,\fb)$. We write $\pol(\fa)$ for $\pol(\fa,\fa)$.
	
	We write $\proj$ for the set of all projections on $2=\{0,1\}$ (which is clearly a minion).
\end{definition}

	Note that the set of polymorphisms of a promise template always forms a minion. The main correspondence between pp-constructions and minion homomorphisms is the following.

\begin{theorem}[\cite{barto2021algebraic}, Theorem 4.12]\label{minions}
	Let $(\fa,\fb)$ and $(\fc,\fd)$ be finite promise templates. Then the following are equivalent.
\begin{enumerate}
\item $(\fa,\fb)$ pp-constructs $(\fc,\fd)$.
\item There exist a minion homomorphism from $\pol(\fa,\fb)$ to $\pol(\fc,\fd)$.
\end{enumerate}
\end{theorem}

	Note that Theorem~\ref{minions} also implies that the pp-constructability relation is transitive which is not immediate from our definition. We mention that the non-promise version of Theorem~\ref{minions} (i.e. when $\fa=\fb$ and $\fc=\fd$) was already shown in~\cite{wonderland}.

\begin{definition}
	We say that a promise template $(\fa,\fb)$ is \emph{omniexpressive} if $(\fa,\fb)$ pp-constructs all finite promise templates. We say that $\fa$ is \emph{omniexpressive} if $(\fa,\fa)$ is.
\end{definition}

	Using this notion, an important special case of Theorem~\ref{minions} is the following.
	
\begin{corollary}\label{hardness}
	Let $(\fa,\fb)$ be a finite promise template. Then the following are equivalent.
\begin{enumerate}[(1)]
\item\label{it:omni} $(\fa,\fb)$ is omniexpressive.
\item\label{it:k3} $(\fa,\fb)$ pp-constructs $(K_3,K_3)$.
\item\label{it:k3s} $(\fa,\fb)$ pp-constructs $(K_3^*,K_3^*)$ where $K_3^*$ denote the expansion of $K_3$ by all constants.
\item\label{it:mtop} There exists a minion homomorphism from $\pol(\fa,\fb)$ to $\proj$.
\end{enumerate}
\end{corollary}

\begin{proof}
	The implication~\ref{it:omni}$\rightarrow$~\ref{it:k3} is trivial and the implication~\ref{it:mtop}$\rightarrow$~\ref{it:k3} follows from Theorem~\ref{minions}. By Lemma 3.9 in~\cite{wonderland} we know that $K_3$ pp-constructs $K_3^*$ since $K_3$ is a core structure. This implies the equivalence of items~\ref{it:k3} and~\ref{it:k3s}. Finally, the equivalence of items~\ref{it:k3s} and~\ref{it:mtop} follows from Theorem~\ref{minions} again and from the fact that all polymorphisms of $K_3^*$ are projections (see, for instance~\cite{bodirsky2021complexity}, Proposition 6.1.43).	
\end{proof}

\begin{remark}
	The term \emph{omniexpressive} itself was only introduced recently in~\cite{brunar2025janus} but the concept itself already appeared in many different places and forms in the literature, many of which are also cited in this article.
\end{remark}

\begin{remark}\label{logspace}
	We know that if $(\fa,\fb)$ pp-constructs $(\fc,\fd)$ then $\pcsp(\fc,\fd)$ is LOGSPACE-reducible to $\pcsp(\fc,\fd)$. Since $\csp(K_3)$, i.e., 3-colouring of graphs is {\npc}, this means that omniexpressivity implies {\nph}ness of the corresponding PCSP. By the finite-domain CSP dichotomy theorem~\cite{BulatovFVConjecture,zhuk2020proof} we know that if $\fa$ is not omniexpressive then $\csp(\fa)$ is tractable, but this is no longer the case in the promise setting, as we will see in the next subsection.
\end{remark}

\subsection{Minor conditions}

	A further characterization of the existence of minion homomorphisms, and in turn pp-constructability, can be given in terms of minor conditions.
	
\begin{definition}
	A \emph{minor condition} is a finite set $\Sigma$ of identities of the form $f_{\sigma}\approx g_{\tau}$ where $f$ and $g$ are functional symbols with some designated arities $k$ and $\ell$, respectively, and $\sigma\in \prescript{k}{}{n},\tau\in \prescript{\ell}{}{n}$ for some $n\in \omega$. 
	
	A minor condition $\Sigma$ is satisfied in a minion $\mathcal{M}$ if and only if we can find an assignment $f\mapsto f^{\mathcal{M}}\in \mathcal{M}$ for symbols $f$ occurring in $\Sigma$ such that whenever $f_{\sigma}\approx g_{\tau}$ is an identity in $\Sigma$ then $f_{\sigma}^{\mathcal{M}}=g_{\tau}^{\mathcal{M}}$ holds.
	
	A minor condition is called \emph{trivial} if it is satisfied by $\proj$. 
\end{definition}

\begin{remark}
Formally speaking, the number $n$ is not specified by the identity $f_{\sigma}\approx g_{\tau}$ as in the first line above, but the choice of its value does not change the satisfiability of $\Sigma$ because we can add or drop dummy variables in $f_{\sigma}^{\mathcal{M}}$ or $g_{\tau}^{\mathcal{M}}$ for each of the identities $f_{\sigma}\approx g_{\tau}$.
\end{remark}

	The correspondence between minion homomorphisms and minor conditions in the case of finite-domain minions is the following.
	
\begin{theorem}[\cite{barto2021algebraic}, Theorem 4.12]\label{minors}
	Let $\mathcal{M}$ and $\mathcal{N}$ be minions with finite domains. Then there exists a minion homomorphism from $\mathcal{M}$ to $\mathcal{N}$ if and only if all minor conditions satisfied by $\mathcal{M}$ are also satisfied by $\mathcal{N}$.
\end{theorem}

	It is worth pointing out that ``only if'' direction in Theorem~\ref{minors} follows from a routine calculation (and it does not even require finiteness), so the real strength of this theorem lies in the other direction. In practice, the non-existence of a minion homomorphism from one minion to another (and in turn non-pp-constructability of structures) is often shown by finding a minor condition which separates the two minions. Theorem~\ref{minors} tells us that this is always possible in the case of finite-domain minions. One of the special cases of this observation is the following.
	
\begin{corollary}\label{hardness2}
	Let $(\fa,\fb)$ be a finite promise template. Then $(\fa,\fb)$ is omniexpressive if and only if all minor conditions satisfied by $\pol(\fa,\fb)$ are trivial.
\end{corollary}

\begin{proof}
	Direct consequence of Corollary~\ref{hardness} and Theorem~\ref{minors}.
\end{proof}

	In this paper we are considering the following types of operations defined by minor conditions.

\begin{definition}\label{def:minor}
	A map $f:A^k\rightarrow B$ is called
\begin{itemize}
\item \emph{cyclic} if $k\geq 2$, and it satisfies the identity $f(x_0,\dots,x_{k-1})=f(x_1,\dots,x_{k-1},x_0)$,
\item a \emph{area-rare} (or \emph{4-ary Siggers}) operation if $k=4$ and it satisfies the identity $f(a,r,e,a)=f(r,a,r,e)$,
\item a \emph{Siggers} operation if $k=6$, and it satisfies the identity $f(x,y,x,z,y,z)=f(y,x,z,x,z,y)$,
\item an \emph{Ol\v{s}\'{a}k} operation if $k=6$, and it satisfies the minor condition $f(x,x,y,y,y,x)=f(x,y,x,y,x,y)=f(y,x,x,x,y,y)$.
\end{itemize}
\end{definition}

	It is easy to check that all minor conditions in Definition~\ref{def:minor} are nontrivial, thus by Corollary~\ref{hardness2} the existence of any of the operations above in $\pol(\fa,\fb)$ implies that $(\fa,\fb)$ is not omniexpressive.
	
	The following proposition is folklore, for the sake of completeness we provide a proof.
	
\begin{proposition}\label{cycl}
	Let $\mathcal{M}$ be any minion. Then
\begin{enumerate}[(1)]
\item\label{it:cyc_area} If $\mathcal{M}$ contains a cyclic operation of any arity then $\mathcal{M}$ also contains a area-rare operation.
\item\label{it:area_so} If $\mathcal{M}$ contains an area-rare operation then it contains both a Siggers and an Ol\v{s}\'{a}k operation.
\end{enumerate}
\end{proposition}

\begin{proof}
	(\ref{it:cyc_area}) Let $f\in \mathcal{M}$ be a cyclic operation of arity $3k+r$ where $r\in \{0,1,2\}$. If $r=0$ then let $$g(x,y,z,w)\coloneqq f(\underbrace{y,\dots,y}_{k \text{ times}},\underbrace{z,\dots,z}_{k \text{ times}},\underbrace{w,\dots,w}_{k \text{ times}}).$$ Then we have
\[
g(a,r,e,a)=f(\underbrace{r,\dots,r}_{k \text{ times}},\underbrace{e,\dots,e}_{k \text{ times}},\underbrace{a,\dots,a}_{k \text{ times}})=f(\underbrace{a,\dots,a}_{k \text{ times}},\underbrace{r,\dots,r}_{k \text{ times}},\underbrace{e,\dots,e}_{k \text{ times}})=g(r,a,r,e).
\]
	If $r=1$ then $k\geq 1$. In this case let $$g(x,y,z,w)\coloneqq f(\underbrace{y,\dots,y}_{k+1 \text{ times}},\underbrace{w,\dots,w}_{k-1 \text{ times}},x,x,\underbrace{z,\dots,z}_{k-1 \text{ times}}).$$ Then
\begin{align*}
g(a,r,e,a)=&f(\underbrace{r,\dots,r}_{k+1 \text{ times}},\underbrace{a,\dots,a}_{k+1 \text{ times}},\underbrace{e,\dots,e}_{k-1 \text{ times}})\\
=&f(\underbrace{a,\dots,a}_{k+1 \text{ times}},\underbrace{e,\dots,e}_{k-1 \text{ times}},\underbrace{r,\dots,r}_{k+1 \text{ times}})=g(r,a,r,e).
\end{align*}
	Finally, if $r=2$ then let $$g(x,y,z,w)\coloneqq f(\underbrace{y,\dots,y}_{k+1 \text{ times}},\underbrace{w,\dots,w}_{k \text{ times}},x,\underbrace{z,\dots,z}_{k \text{ times}}).$$ Then
\begin{align*}
g(a,r,e,a)=&f(\underbrace{r,\dots,r}_{k+1 \text{ times}},\underbrace{a,\dots,a}_{k+1 \text{ times}},\underbrace{e,\dots,e}_{k \text{ times}})\\
=&f(\underbrace{a,\dots,a}_{k+1 \text{ times}},\underbrace{e,\dots,e}_{k \text{ times}},\underbrace{r,\dots,r}_{k+1 \text{ times}})=g(r,a,r,e).
\end{align*}

	(\ref{it:area_so}) Let $f$ be an operation satisfying the identity $f(a,r,e,a)=f(r,a,r,e)$. Then $g(x,y,z,u,v,w)\coloneqq g(x,y,w,z)$ is a Siggers operation and $h(x,y,z,u,v,w)\coloneqq g(v,y,x,z)$ is an Ol\v{s}\'{a}k operation as shown by the following calculation.
\begin{align*}
	g(x,y,x,z,y,z)&=f(x,y,z,x)=f(y,x,y,z)=g(y,x,z,x,z,y),\\
	h(x,x,y,y,y,x)&=f(y,x,x,y)=f(x,y,x,x)=h(x,y,x,y,x,y),\\
	h(x,y,x,y,x,y)&=f(x,y,x,x)=f(y,x,y,x)=h(y,x,x,x,y,y).\qedhere
\end{align*}
\end{proof}

	We mention that in the case of a single finite structure all the minor conditions above are equivalent. In fact, the following holds.
	
\begin{theorem}\label{minor_equiv}
	Let $\fa$ be a finite relational structure. Then the following are equivalent.
\begin{enumerate}[(1)]
\item\label{it:notomni} $\fa$ is not omniexpressive.
\item\label{it:nontrivi} $\pol(\fa)$ satisfies a nontrivial minor condition.
\item\label{it:siggers} $\fa$ has a Siggers polymorphism.
\item\label{it:olsak} $\fa$ has an Ol\v{s}\'{a}k polymorphism.
\item\label{it:area} $\fa$ has an area-rare polymorphism.
\item\label{it:cyclic} $\fa$ has a cyclic polymorphism.
\end{enumerate}
\end{theorem}

\begin{proof}
	
	The equivalence of items~\ref{it:notomni} and~\ref{it:nontrivi} follows from Corollary~\ref{hardness2}. The implications~\ref{it:cyclic}$\rightarrow$~\ref{it:area},~\ref{it:area}$\rightarrow$~\ref{it:olsak},~\ref{it:area}$\rightarrow$~\ref{it:siggers} follow from Proposition~\ref{cycl}, and the implications~\ref{it:olsak}$\rightarrow$~\ref{it:nontrivi} and~\ref{it:siggers}$\rightarrow$~\ref{it:nontrivi} are trivial.
	
	Finally, the implication~\ref{it:notomni}$\rightarrow$~\ref{it:cyclic} has been shown in~\cite{Cyclic} in the case when $\pol(\fa)$ is idempotent, i.e., it preserves all constant relations on $A=\dom(\fa)$. However, the idempotency assumption can be removed by some general arguments, see for instance Section 6.9 in~\cite{bodirsky2021complexity}.
\end{proof}

\begin{remark}
	The equivalence of items~\ref{it:notomni} and~\ref{it:siggers} in Theorem~\ref{minor_equiv} was already observed in~\cite{Siggers}. Ol\v{s}\'{a}k operations originally appeared in~\cite{olsak-idempotent} where it was shown that an idempotent algebra is Taylor if and only if it contains an Ol\v{s}\'{a}k term. From this, one can also easily derive the equivalence of items~\ref{it:siggers} and~\ref{it:olsak} above.
\end{remark}


	The following theorem summarises many of the known results about the existence of Siggers and Ol\v{s}\'{a}k polymorphisms of certain promise templates.
	
\begin{theorem}\label{graph_promise}
	Let $k,c\in \omega$ with $2\leq k\leq c$. Then the following hold.
\begin{enumerate}[(1)]
\item\label{it:komni} The promise template $(K_k,K_c)$ omniexpressive if and only if $3\leq k\leq c\leq 2k-2$.
\item\label{it:kolsak} The promise template $(K_k,K_c)$ has an Ol\v{s}\'{a}k polymorphism if and only if $c\geq 2k$ or $k=2$.
\item\label{it:ksiggers} The promise template $(K_k,K_c)$ has a Siggers polymorphism if and only if $k=2$.
\item\label{it:psiggers} A finite promise template has no Siggers polymorphism if and only if it pp-constructs $(K_3,K_d)$ for some $3\leq d\in \omega$.
\item\label{it:holsak} The promise template $(H_k,H_c)$ has no Ol\v{s}\'{a}k polymorphism.
\item\label{it:polsak} A finite promise template has no Ol\v{s}\'{a}k polymorphism if and only if it pp-constructs $(H_2,H_d)$ for some $2\leq d\in \omega$.
\end{enumerate}
\end{theorem}

\begin{proof}
	It is well-known that $K_2$ is not omniexpressive, thus by Theorem~\ref{minor_equiv} it has both a Siggers and an Ol\v{s}\'{a}k polymorphism. Alternatively, one can check that the ternary majority operation $\prescript{3}{}{2}\rightarrow {2}$ (i.e., the operation which outputs the only value which appears at least twice in the input) is a cyclic polymorphism of $K_2$, and then we can arrive to the same conclusion by using Proposition~\ref{cycl}. This settles items~\ref{it:komni},~\ref{it:kolsak},~\ref{it:ksiggers} in the case when $k=2$.
	
	In the case when $k\geq 3$ item~\ref{it:kolsak} follows from Lemma 6.4 and Proposition 10.1 in~\cite{barto2021algebraic} and item~\ref{it:komni} follows from the results of~\cite{brakensiek2016new}.
	
	Item~\ref{it:ksiggers} for $k\geq 2$ and item~\ref{it:psiggers} follow from Theorem 6.9 in~\cite{barto2021algebraic}.
	
	Since $(H_2,H_c)$ is a homomorphic relaxation of $(H_k,H_c)$ it is enough to show item~\ref{it:holsak} in the case when $k=2$. Both this and item~\ref{it:polsak} follow from Theorem 6.2 in~\cite{barto2021algebraic}.
\end{proof}

\begin{corollary}\label{graph_promise2}
	For all $k\geq 3$ there exists some $d$ such that $(K_k,K_{2k-1})$ pp-constructs $(H_2,H_d)$.
\end{corollary}

	It was shown in~\cite{dinur2005hardness} that $\csp(H_2,H_c)$ is {\npc} for any $2\geq k\leq c$. By Corollary~\ref{graph_promise2} and Remark~\ref{logspace} this also implies that hardness of $\pcsp(K_k,K_{2k-1})$, as it was pointed out in~\cite{barto2021algebraic}. In~\cite{krokhin2023topology} it was shown by different methods that in fact $\pcsp(K_k,K_c)$ is {\npc} for all $3\leq k\leq c\leq {k\choose \lfloor k/2 \rfloor}-1$ (which improves on the previous result if $k\geq 5$).

\subsection{Finite versions of the axiom of choice}

	Let us consider the following axioms.

\begin{itemize}
\item $\ac$ (Axiom of choice). For every family of nonempty sets $X$ there is a function $X\rightarrow \bigcup X$ such that $f(x)\in x$ for all $x\in X$.
\item $\ufl$ (Ultrafilter lemma). Every filter is contained in an ultrafilter.
\item $\ac(<\!\omega)$ (Axiom of finite choice). For every family of nonempty finite sets $X$, there is a function $X\rightarrow \bigcup X$ such that $f(x)\in x$ for all $x\in X$.
\item $\ac(k)$. For every family $X$ of sets of size exactly $k$ there is a function $X\rightarrow \bigcup X$ such that $f(x)\in x$ for all $x\in X$.
\item $\acf$. For all $k\in \omega$, $\ac(k)$ holds.
\item $\kw$ (Kinna-Wagner principle). For every family $X$ of sets of size at least 2 there is a function $X\rightarrow \mathcal{P}(\bigcup X)$ such that for all $x\in X$, $f(x)$ is a proper subset of $x$.
\item $\kw(k)$. For every family $X$ of sets of size exactly $k$ there is a function $X\rightarrow \mathcal{P}(\bigcup X)$ such that for all $x\in X$, $f(x)$ is a proper subset of $x$.
\end{itemize}


	We know that over $\zf$

\begin{enumerate}[(i)]
\item $\ac$ is strictly stronger than $\ufl$,
\item $\ufl$ is strictly stronger than $\ac(<\!\omega)$, and
\item $\ac(<\!\omega)$ is strictly stringer than $\acf$,
\end{enumerate}

see for instance~\cite{jech1977axiom}.

	
	Clearly, $\ac(k)$ implies $\kw(k)$. Note that in formulation of the axiom $\kw(k)$ we can assume without loss of generality that $|f(x)|\leq k/2$ for all $x$ since for all values of size bigger than $k/2$ we can switch to the complement. From this observation it is easy to see that $\kw(k)\wedge \bigwedge_{\ell\leq k/2}\kw(\ell)$ implies $\ac(k)$ for all $k\in \omega$, in particular $\kw(k)$ and $\ac(k)$ are equivalent for $k=2$ and $k=3$. It is also easy to see that $\ac(kn)$ implies $\ac(k)$ for all $k,n\in \omega$. For prime numbers the converse also holds in the following sense. For any prime $p$ there is a model of $\zf$ in which $\ac(p)$ fails but $\ac(k)$ holds for all $k$ not divisible by $p$. In fact, the exact pairs of sets $(Z,k)\in \mathcal{P}(\omega)\times \omega$ for which the implication $\bigwedge_{\ell\in Z}\ac(\ell)\rightarrow \ac(k)$ is provable in ZF have been classified~\cite{gauntt}.
	From this we only need the following nontrivial implication in this paper.
	
\begin{theorem}[\cite{jech1977axiom}, Theorem 7.15]\label{prime_sum}
	Let $n,m\in \omega$ and let us assume that $n$ cannot be written as a sum of primes bigger than $m$. Then $\bigwedge_{k\leq m}\ac(k)$ implies $\ac(n)$.
\end{theorem}

	Theorem~\ref{prime_sum} immediately implies the following.

\begin{corollary}\label{prime_smallest}
	The smallest $k$ for which $\ac(k)$ fails (if there is any) is prime.
\end{corollary}

\begin{corollary}\label{kw_prime}
	The following are equivalent.
\begin{enumerate}[(1)]
\item\label{it:acf} $\acf$.
\item\label{it:acfp} $\ac(p)$ holds for every prime number $p$.
\item\label{it:kwp} $\kw(p)$ holds for every prime number $p$.
\end{enumerate}
\end{corollary}

\begin{proof}
	The implications~\ref{it:acf}$\rightarrow$~\ref{it:acfp} and~\ref{it:acfp}$\rightarrow$~\ref{it:kwp} are trivial. 
	
	Now let us assume that item~\ref{it:kwp} holds but item~\ref{it:acf} does not. By Corollary~\ref{prime_smallest} we know that there exists a prime $p$ such that $\bigwedge_{k< p}\ac(k)$ holds but $\ac(p)$ does not. By our assumption we know that $\kw(p)$ holds, thus $\bigwedge_{k\leq p}\kw(k)$ also holds. However, as we have seen, this implies $\ac(p)$, a contradiction.
\end{proof}

\section{Promise compactness}\label{sect:promise}

	In this section we introduce the promise version of compactness principles and examine some their basic properties. In particular, we show that pp-constructions lead to implications between the compactness principles over the corresponding templates.
	
\begin{definition}
	Let $(\fa,\fb)$ be a finite promise template. Then we write $K_{(\fa,\fb)}$, called the \emph{compactness principles} over $(\fa,\fb)$, for the following statement.\\ 
	\emph{For every structure $\ii$ with the same signature as $(\fa,\fb)$ if for every finite substructure of $\ii$ admits a homomorphism to $\fa$ then $\ii$ admits a homomorphism to $\fb$.}
	
	 We write $K_{\fa}$ for $K_{(\fa,\fa)}$, and we call it the \emph{compactness principle} over $\fa$.
\end{definition}

	
	Note that all compactness principles $K_{(\fa,\fb)}$ are theorems of $\zfc$. In this paper, however, all results are understood over $\zf$, unless indicated otherwise.

	In~\cite{katay2023csp} it is shown that if $\fa$ pp-constructs $\fb$ then $K_{\fa}$ implies $K_{\fb}$ over $\zf$. We will show that this result generalises to promise templates as well.
	
	The following definition is motivated by Definition 2.6 in~\cite{katay2023csp}.
	 
	
	
\begin{definition}\label{def:fin}
	Let $(\fa,\fb)$ and $(\fc,\fd)$ be finite promise templates. Then we say that $(\fc,\fd)$ \emph{finitely reduces} to $(\fa,\fb)$ if there exists an operation $\Gamma$ mapping structures similar to $\fc$ to structures similar to $\fa$ such that

	
\begin{enumerate}[(i)]
\item\label{it:def1} for every structure $\ii$ similar to $\fc$ if $\ii\rightarrow \fc$ then $\Gamma(\ii)\rightarrow \fa$,
\item\label{it:def2} for every structure $\ii$ similar to $\fc$ if $\Gamma(\ii)\rightarrow \fb$ then $\ii\rightarrow \fd$, and
\item\label{it:def3} if there exists a finite substructure $\fh$ of $\Gamma(\ii)$ that does not admit a homomorphism to $\fa$ then there exists a finite substructure of $\ff$ of $\ii$ that does not admit a homomorphism to $\fc$.
\end{enumerate}
\end{definition}

\begin{remark}
	We remark that our definition of finite reductions slightly differs from the one given in~\cite{katay2023csp} in that in the latter it was also assumed that there exist concrete operations mapping homomorphisms to homomorphisms as in item~\ref{it:def1}. However, this strong version is not required for any of our arguments regarding compactness statements.
\end{remark}

	Similarly to the argument in~\cite{katay2023csp} we can show the following.
	
\begin{lemma}\label{lem:fin}
	Let us assume that $\fa,\fb,\fc,\fd,\Gamma$ be as in item~\ref{it:def1} in Definition~\ref{def:fin}, and let us assume that for every finite substructure $\fh$ of $\Gamma(\ii)$ there exists a finite structure $\ff$ of $\ii$ such that $\fh\rightarrow\Gamma(\ff)$. Then item~\ref{it:def3} in Definition~\ref{def:fin} also holds.
\end{lemma}

\begin{proof}
	Let $\ii$ be a structure similar to $\fc$, and let $\fh$ be a finite substructure of $\Gamma(\ii)$ that does not admit a homomorphism to $\fa$. We know that there exists some finite substructure $\ff$ of $\ii$ such that $\fh\rightarrow\Gamma(\ff)$. Then $\Gamma(\ff)\nrightarrow \fa$, and thus $\ff\nrightarrow \fc$.
\end{proof}

	
\begin{lemma}\label{compact}
	Let $(\fa,\fb)$ and $(\fc,\fd)$ be finite promise templates such that $(\fc,\fd)$ finitely reduces to $(\fa,\fb)$. Then $K_{(\fa,\fb)}$ implies $K_{(\fc,\fd)}$.
\end{lemma}

\begin{proof}
	Let $\Gamma$ be as in Definition~\ref{def:fin}. 
	
	Let us assume that $K_{(\fa,\fb)}$ holds, and let $\ii$ be a structure similar to $\fc$ such that every finite substructure $\ff$ of $\ii$ admits a homomorphism to $\fc$. Then by item~\ref{it:def3} in Definition~\ref{def:fin} we know that every finite substructure of $\Gamma(\ii)$ admits a homomorphism to $\fa$. Thus, $\Gamma(\ii)\rightarrow \fb$ and $\ii\rightarrow \fd$ by item~\ref{it:def2} in Definition~\ref{def:fin}
\end{proof}

\subsection{Compactness from pp-constructions}

	In this subsection we show that pp-constructability implies finite reducability for promise templates.

\begin{lemma}\label{relax}
	If $(\fc,\fd)$ is a homomorphic relaxation of $(\fa,\fb)$ then $(\fc,\fd)$ \emph{finitely reduces} to $(\fa,\fb)$.
\end{lemma}

\begin{proof}
	Let $f: \fc\rightarrow \fa$ and $g: \fb\rightarrow \fd$ be homomorphisms. Then we set $\Gamma$ to be the identity map and $\Phi(\varphi)\coloneqq f\circ \varphi$ and $\Psi(\psi)\coloneqq g\circ \psi$. It is straightforward to verify that this choice satisfies the conditions of Definition~\ref{def:fin}.
\end{proof}

\begin{lemma}\label{ppp}
	Let $(\fc,\fd)$ be a pp-power of $(\fa,\fb)$. Then $(\fc,\fd)$ finitely reduces to $(\fa,\fb)$.
\end{lemma}

\begin{proof}
	The proof for this is essentially the same as the proof of~\cite{katay2023csp} which treats the special case when $\fa=\fb$.

	Let us assume that $(\fc,\fd)$ is the $n$th pp-power of $(\fa,\fb)$, and let us consider the formulas $\varphi_R$ as in Definition~\ref{ppp}. Now let us consider the map $\Gamma$ mapping the instances of $(\fa,\fb)$ to $(\fc,\fd)$ as defined in the proof of Theorem 2.12 in~\cite{katay2023csp} using the formulas $\varphi_R$. This makes sense since the definition of $\Gamma$ in this proof is completely syntactical, i.e., it does not depend on the template structures. It is shown in the aforementioned proof that $\Gamma$ gives a finite reduction from $(\fc,\fc)$ to $(\fa,\fa)$ and from $(\fd,\fd)$ to $(\fb,\fb)$. Therefore, $\Gamma$ also gives a finite reduction from $(\fc,\fd)$ to $(\fa,\fb)$.
\end{proof}

	Lemmas~\ref{relax} and~\ref{ppp} together imply the following.
	
\begin{theorem}\label{construct_reduce}
	If $(\fa,\fb)$ pp-constructs $(\fc,\fd)$ then $(\fc,\fd)$ finitely reduces to $(\fa,\fb)$.
\end{theorem}

\begin{corollary}\label{construct_reduce2}
	If $(\fa,\fb)$ pp-constructs $(\fc,\fd)$ then $K_{(\fa,\fb)}$ implies $K_{(\fc,\fd)}$. In particular, if $(\fa,\fb)$ is omniexpressive then $K_{(\fa,\fb)}$ is equivalent to $\ufl$.
\end{corollary}

\begin{proof}
	The first statement follows from~\ref{construct_reduce} and Lemma~\ref{compact}. The second statement follows from the fact that $\ufl$ is equivalent to $K_{K_3}$ and $K_3$ is omniexpressive.
\end{proof}

\begin{corollary}\label{promise_compactness}
	For all $3\leq k\leq c \leq 2k-2$ the compactness principle $K_{(K_k,K_c)}$ is equivalent to $\ufl$.
\end{corollary}

\begin{proof}
	Follows from Theorem~\ref{graph_promise} and Corollary~\ref{construct_reduce2}.
\end{proof}

	In Section~\ref{sect:weak} we will show that in fact $c$ can be increased to $2k-1$ in the statement Corollary~\ref{promise_compactness}.



\subsection{Compactness principles provable in $\zf$}

	We close this section off by characterizing those finite promise templates $(\fa,\fb)$ for which $K_{(\fa,\fb)}$ is provable in $\zf$. By the results of~\cite{rorabaugh2017logical} and~\cite{tardif2025constraint} we know that for a finite structure $\fa$ the compactness principle $K_{\fa}$ is a theorem of $\zf$ if and only if the $\fa$ has \emph{width 1}. In this subsection we show that the analogous characterization also holds for finite promise templates.
	
	We first recall the definition of width 1 for promise templates.
	
\begin{definition}
	Let $\fa$ be an arbitrary relational structure. Then we denote by $\upower(\fa)$, called the \emph{power structure of $\fa$}, the relational structure similar to $\fa$ whose domain set is $\power(\fa)\setminus \{\emptyset\}$ and for each relational symbol $R$ of arity $k$ we have $(S_0,\dots,S_{k-1})\in R^{\upower(\fa)}$ if and only if for all $j\in k$ and $a_i\in S_i$ there exists a $b_j\in S_j$ such that $(a_0,\dots,a_{j-1},b_j,a_{j+1},\dots,a_k)\in R^{\fa}$.
	
	We say that a promise template is $(\fa,\fb)$ has \emph{width 1} if $\upower(\fa)$ homomorphically maps to $\fb$.
\end{definition}

	We note that width 1 templates are usually defined differently, for instance by solvability the corresponding PCSP by arc consistency, however the equivalence to our definition can be seen easily, see the discussion in~\cite{barto2021algebraic}, Section 7.

	We write $\fh$ for the structure with $\dom(\fh)=2$ and whose relations are $$\{0\}, \{1\}, \{(x,y,z)\in \prescript{3}{}{2}: x\wedge y\rightarrow z\}, \{(x,y,z)\in \prescript{3}{}{2}: x\wedge y\rightarrow \neg z\}.$$ (The CSP over $\fh$ is usually called \emph{Horn-3-SAT}.) Using this, width 1 promise templates can be characterised as follows.
	
\begin{theorem}[\cite{barto2021algebraic}, Theorem 7.4]\label{width1}
	Let $(\fa,\fb)$ be a finite promise template. Then the following are equivalent.
\begin{enumerate}
\item $(\fa,\fb)$ has width 1.
\item $(\fh,\fh)$ pp-constructs $(\fa,\fb)$.
\end{enumerate}
\end{theorem}

\begin{corollary}\label{true}
	Let $(\fa,\fb)$ be a finite promise template which has width 1. Then $K_{(\fa,\fb)}$ is provable in $\zf$.
\end{corollary}

\begin{proof}
	The case when $\fa=\fb$ is shown in~\cite{rorabaugh2017logical}. In particular, $K_{\fh}$ is provable in $\zf$.
	
	Now let $(\fa,\fb)$ be an arbitrary finite promise template. Then by Theorem~\ref{width1} we know that $(\fh,\fh)$ pp-constructs $(\fa,\fb)$. Therefore, by Lemma~\ref{compact} it follows that $K_{(\fa,\fb)}$ is provable in $\zf$.
\end{proof}

\begin{remark}
	We remark that the proof of~\cite{rorabaugh2017logical} can also be adapted to the promise setting in a straightforward way which would give an alternative proof for Corollary~\ref{true}.
\end{remark}

	In~\cite{tardif2025constraint} it is shown that if $\fa$ does not have width 1, i.e., there is no homomorphism from $\upower(\fa)$ to $\fa$, then $K_{\fa}$ implies the existence a non-measurable subset of $\mathbb{R}^3$, and therefore it is not provable in $\zf$. Note, however, that the proof of Lemma 4.2 and Theorem 1.2 in the aforementioned paper also works word for word if we assume $K_{(\fa,\fb)}$ instead of $K_{\fa}$, and we replace all occurrences of the structure $\fa$ with $\fb$, while keeping all occurrences of $\upower(\fa)$. Therefore, we obtain the following.
	
\begin{theorem}[$\zf$]\label{not_true}
	Let $(\fa,\fb)$ be a finite promise template which does not have width 1 and suppose that $K_{(\fa,\fb)}$ holds. Then there exists a non-measurable subset of $\mathbb{R}^3$.
\end{theorem}

\begin{corollary}[$\zf$]\label{when_true}
	Let $(\fa,\fb)$ be a finite promise template. Then $K_{(\fa,\fb)}$ is provable in $\zf$ if and only if $(\fa,\fb)$ has width 1.
\end{corollary}

\begin{proof}
	Follows from Corollary~\ref{true}, Theorem~\ref{when_true} and from the fact that there exist models of $\zf$ where all subsets of $\mathbb{R}^3$ are measurable.
\end{proof}

\section{Finite choice from the absence of cyclic polymorphisms}\label{sect:ac_fin}

	In this section we show that if a promise template does not have a cyclic polymorphism then the corresponding compactness principle implies the axiom $\acf$. 
	
	Our proof relies on the following more specific lemma.

\begin{lemma}\label{thm:finite}
    Let $p$ be a prime and let $(\fa,\fb)$ be a finite promise template which does not have a cyclic polymorphism of arity $p$. Then $K_{(\fa,\fb)}$ implies $\kw(p)$.
\end{lemma}

\begin{proof}
    Let $X$ be family of sets of size $p$. We will find a function $X\rightarrow \mathcal{P}(\bigcup X)$ such that $f(C)$ is a proper subset of $C$ for all $C\in X$. We can assume without loss of generality that the elements of $X$ are pairwise disjoint. We define a structure $\fm$ with $\dom(\fm)=M\coloneqq \bigcup_{C\in X}\prescript{C\!}{}{A}$ similar to $\fa$ as follows. Let $R$ be a relation of $\fa$ of arity $k$. Then for each $C\in X$ the tuple $((x_c^1)_{c\in C},\dots,(x_c^k)_{c\in C})\in (\prescript{C\!}{}{A})^k$ is contained in $R^{\fm}$ if and only if $(x_c^1,\dots,x_c^k)\in R$ for all $c\in C$. No other tuples are in the relation $R^{\fm}$. Note that $\fm$ is a disjoint union of copies of $\fa^p$. If $F\subseteq M$ is finite then $F$ is contained in a structure isomorphic to some finite disjoint union of copies $\fa^p$. Clearly, these structures homomorphically map to $\fa$, for instance we can take any projection map on each copy of $\fa^p$. Therefore, by our assumption, $\fm$ admits a homomorphism $h$ to $\fb$. 

    Now, let $\prec$ be any ordering of the $p$-ary polymorphisms of $(\fa,\fb)$. Such an ordering exists since both $A$ and $B$, and thus also $\prescript{\prescript{p\!}{}A\!}{}{B}$ are finite. Now let $C\in X$ be arbitrary. We define $f(C)$ as follows. For a map $\alpha: C \rightarrow p$ and $f\in \prescript{\prescript{C\!}{}A\!}{}{B}$ we write $f_{\alpha}$ for the map from $\prescript{p\!}{}{A}\rightarrow B$ which maps each tuple $x\in \prescript{p\!}{}{A}$ to $f(x\circ \alpha)$. Then $({h|}_{\prescript{C\!}{}{A}})_{\alpha}$ is a polymorphism of $(\fa,\fb)$ for all bijections $\alpha$ from $C$ to $p$. We define $g(C)$ to be the $\prec$-smallest polymorphism which can be obtained this way and we define
\[f(C)\coloneqq \{\beta^{-1}(0)\mid \beta\colon C\rightarrow p, ({h|}_{\prescript{C\!}{}{A}})_{\beta}=g(C)\}.\] We show that this $f$ suffices.

    Clearly, $\alpha^{-1}(0)\in f(C)$ for the map $\alpha$ in the definition of $g(C)$. Thus, $f(C)$ is a nonempty subset of $C$ for all $C\in X$. We need to show that $f(C)\neq C$. Let $\Gamma\coloneqq\{\sigma\in \sym(p): g(C)_{\sigma}=g(C)\}$. Then clearly $\Gamma$ is a subgroup of symmetric group $\sym(p)$. We claim that $\Gamma$ is not transitive. Suppose on the contrary that $\Gamma$ is transitive. Then the stabiliser of each element is an index $p$ subgroup of $\Gamma$. In particular $p\mid|\Gamma|$. Thus, by Cauchy's theorem it follows that $\Gamma$ contains some element $\gamma$ of order $p$. This is only possible if $\gamma$ is a $p$-cycle, say $\gamma=(i_0\,\dots\,i_{p-1})$. Let $g'\coloneqq (x_0,\dots,x_{p-1})\mapsto g(C)(x_{i_0},\dots,x_{i_{p-1}})$. Then we have
\begin{align*}
g'(x_1,\dots,x_{p-1},x_0)=&g(C)(x_{i_1},\dots,x_{i_{p-1}},x_{i_0})\\=&g(C)(x_{\gamma(i_0)},\dots,x_{\gamma(i_{p-1})})\\=&g(C)(x_{i_0},\dots,x_{i_{p-1}})=g'(x_0,\dots,x_{p-1})
\end{align*} where the second-to-last equality follows from the definition of $\Gamma$. This implies that $g'$ is a cyclic polymorphism of $(\fa,\fb)$ (of arity $p$) which contradicts our assumption.

    Now let $\alpha$ be a bijection $p\rightarrow C$ as in the definition of $g(C)$. Then we have 
\begin{align*}
f(C)=&\{\beta^{-1}(0)\mid \beta\colon C\rightarrow p, ({h|}_{\prescript{C\!}{}{A}})_{\beta}=g(C)\}\\=&\{\alpha^{-1}(\sigma^{-1}(0))\mid \sigma\in \sym(p), ({h|}_{\prescript{C\!}{}{A}})_{\sigma\circ \alpha}=g(C)\}\\
=&\{\alpha^{-1}(\sigma^{-1}(0))\mid \sigma\in \sym(p), g(C)_\sigma=g(C)\}\\
=&\alpha^{-1}(\Gamma(0)).
\end{align*}

    Since $\Gamma$ is not transitive, this implies that $f(C)\neq \alpha^{-1}(p)=C$ which finishes the proof of the theorem.
\end{proof}

	Let us denote by $\vec C_n$ the directed $n$-cycle digraph for $n\in \omega$. It was shown in~\cite{rorabaugh2017logical}, Proposition 6.2 that $K_{\vec C_p}$ implies $\kw(p)$ for any prime number $p$. This now also follows from our general result using the following observation.

\begin{proposition}
	Let $\fg$ be a finite loopless digraph such that $\vec C_p\rightarrow \fg$. Then $(C_p,\fg)$ does not have a cyclic polymorphism of arity $p$.
\end{proposition}

\begin{proof}
	Let us assume for contradiction that $f\in \pol(\vec C_p,\fg)$ is cyclic and of arity $p$. Then since $f$ is a polymorphism, it follows that $f(0,1,\dots,p-1)$ is connected to $f(1,\dots,p-1,0)$ in $\fg$. On the other hand, by the cyclicity of $f$ we get $f(0,1\dots,p-1)=f(1,\dots,p)$ which would mean that $\fg$ has a loop.
\end{proof}

\begin{corollary}\label{cycle}
	Let $\fg$ be a finite loopless digraph such that $\vec C_p\rightarrow \fg$. Then $K_{(\vec C_p,\fg)}$ implies $\kw(p)$.
\end{corollary}

	In the rest of the paper we are only interested in the case when Lemma~\ref{thm:finite} can be applied to all prime numbers $p$.

\begin{theorem}\label{cor:finite}
    Let $(\fa,\fb)$ be a finite promise template which does not have a cyclic polymorphism. Then $K_{(\fa,\fb)}$ implies $\acf$.
\end{theorem}

\begin{proof}
    Follows from Lemma~\ref{thm:finite} and Corollary~\ref{kw_prime}.
\end{proof}

\begin{remark}
	One can notice that in Theorem~\ref{cor:finite} it is enough to assume that $(\fa,\fb)$ does not have a cyclic polymorphism of any prime arity. However, this does not make any difference because in any minion the existence of a cyclic operation also implies the existence of a cyclic operation of prime arity.
\end{remark}

    Note that in the case when $\fa=\fb$ Corollary~\ref{cor:finite} is uninteresting since in this case it follows from Theorem~\ref{minor_equiv} that the assumption of Theorem~\ref{cor:finite} can only hold if $\fa$ is omniexpressive, and in this case we already know by Corollary~\ref{construct_reduce2} that $K_{\fa}$ implies $\ufl$. However, in the promise setting we have several interesting examples for templates without cyclic polymorphisms.

\begin{example}
\begin{itemize}
    \item By item~\ref{it:ksiggers} of Theorem~\ref{graph_promise} we know that if $3\leq k\leq c$ then $(K_k,K_c)$ does not have a Siggers polymorphism, and thus by Proposition~\ref{cycl} it also does not have a cyclic polymorphism.
    \item By item~\ref{it:ksiggers} of Theorem~\ref{graph_promise} we know that if $2\leq k\leq c$ then $(H_k,H_c)$ does not have an Ol\v{s}\'{a}k polymorphism, and thus by Proposition~\ref{cycl} it also does not have a cyclic polymorphism.
\end{itemize} 
\end{example}

\begin{corollary}[$\zf$]\label{cor:ac}
	Let us assume that either $K_{(K_k,K_c)}$ holds for some $3\leq k\leq c$ or $K_{(H_2,H_c)}$ holds for some $c\geq 2$. Then $\acf$ holds.
\end{corollary}

	In Section~\ref{sect:main} we will see that all compactness principles of the form $K_{(H_2,H_c)}: c\geq 2$ are in fact equivalent to $\ufl$. Interestingly, however, our proof for this stronger statement relies on Corollary~\ref{cor:ac}.

\section{Reductions using weak minion homomorphisms}\label{sect:weak}

	In this section we consider a more general notion of minion homomorphisms, called \emph{minion homomorphisms} (or $(d,r)$-minion homomorphisms) which was originally introduced in~\cite{barto2021algebraic}. The main purpose of this notion was to give a general algebraic framework for proving some hardness results for PCSPs which do not follow from omniexpressivity. This includes the templates $(K_k,K_{2k-1})$ for all $k\geq 3$ and $(\hh_2,\hh_c)$ for all $c\geq 2$, see our discussion below. We will show how this algebraic approach can be adapted in the context of compactness principles, and as a result we will obtain that over $\zf$ all compactness principles corresponding to the aforementioned promise templates are equivalent to $\ufl$.

\begin{notation}
	Let $\mathcal{M}$ be a minion. A \emph{chain of minors} is a sequence of minors $$t_0\xrightarrow{\pi_{0,1}}t_1\xrightarrow{\pi_{1,2}}\dots \xrightarrow{\pi_{r-1,r}}t_r$$ in $\mathcal{M}$. For such a sequence for $i<j$ we write $\pi_{i,j}$ for the composition $\pi_{j-1,j}\circ \dots \circ \pi_{i,i+1}$. Note that in this case we have $t_i\xrightarrow{\pi_{i,j}}t_j$.
\end{notation}

\begin{definition}\label{def:min}
	 Let $\mathcal{M},\mathcal{N}$ be minions, and let $d,r\in \omega$. Then a map $\xi: \mathcal{M}\rightarrow \mathcal{P}(\mathcal{N})$ is a \emph{$(d,r)$-minion homomorphism} if
\begin{enumerate}[(i)]
\item $\xi$ preserves arities, i.e., for all $t\in \mathcal{M}$ all elements of $\xi(t)$ have the same arity as $t$,
\item $|\xi(t)|\leq d$ for all $t\in \mathcal{M}$, and
\item for all chains of minors $t_0\xrightarrow{\pi_{0,1}}t_1\xrightarrow{\pi_{1,2}}\dots \xrightarrow{\pi_{r-1,r}}t_r$ in $\mathcal{M}$ there exists $i<j, g\in \xi(t_i), h\in \xi(t_j)$ such that $g\xrightarrow{\pi_{i,j}}h$.
\end{enumerate}

	We say that a map $\xi: \mathcal{M}\rightarrow \mathcal{P}(\mathcal{N})$ is a \emph{weak minion homomorphism} if it is a $(d,r)$-minion homomorphism for some $d,r\in \omega$.
\end{definition}


\begin{remark}
	Note that a minion homomorphism is the same as a $(1,1)$-minion homomorphism.
\end{remark}

	We know by the combination of Theorem~\ref{minions} and Lemma~\ref{construct_reduce} that if there exists a minion homomorphism from $\pol(\fa,\fb)$ to $\pol(\fc,\fd)$ then $(\fc,\fd)$ finitely reduces to $(\fa,\fb)$, and thus by Lemma~\ref{compact} $K_{(\fa,\fb)}$ implies $K_{(\fc,\fd)}$ over $\zf$. When trying to generalize this implication for weak minion homomorphisms we run into several challenges. The first one is that at the moment there is no known corresponding description of weak minion homomorphisms on the level of relational structures (similar to pp-constructions). For this reason, we have to work with the minions directly. Fortunately, the complexity reduction presented in Theorem B2 in~\cite{barto2021algebraic} can easily be adapted to also give a finite reduction. Another slight technical issue is that in these constructions it is not clear how to avoid choice entirely. However, as we will see, we can make our proofs work assuming $\acf$ which will be enough in our concrete applications using Corollary~\ref{cor:ac}.

\begin{theorem}[$\zf+\acf$]\label{dr_reduction}
	Let us assume that there exists a weak minion homomorphism from $\pol(\fa,\fb)$ to $\pol(\fc,\fd)$. Then $(\fc,\fd)$ finitely reduces to $(\fa,\fb)$.
\end{theorem}

	Before proving Theorem~\ref{dr_reduction}, we discuss some of its most important consequences.

	The following result, although not explicitly stated, is included in the proof of Theorem 12 in~\cite{nakajima2025complexity}.

\begin{theorem}\label{wrochna}
	For all $c\geq 2$ there exists a weak minion homomorphism from $\pol(\hh_2,\hh_c)$ to $\proj$. 
\end{theorem}

	Combining this with Theorem~\ref{dr_reduction} we obtain the following.
	
\begin{theorem}[$\zf$]\label{thm:main}
	Let $(\fa,\fb)$ be a finite promise template which does not have an Ol\v{s}\'{a}k polymorphism. Then $K_{(\fa,\fb)}$ is equivalent to $\ufl$.
\end{theorem}

\begin{proof}
	We have already seen that $\ufl$ implies all compactness principles over $\zf$.
	
	For the other direction let $(\fa,\fb)$ be a promise template without an Ol\v{s}\'{a}k polymorphism and assume that $K_{(\fa,\fb)}$ holds.  By item (6) of Theorem~\ref{graph_promise} we know that $(\fa,\fb)$ pp-constructs $(\hh_2,\hh_c)$ for some $c\in \omega,c\geq 2$, and thus by Corollary~\ref{construct_reduce2} we know that $K_{(\hh_2,\hh_c)}$ holds. Then by Corollary~\ref{cor:ac} we know that $\acf$ holds and thus we can apply Theorem~\ref{dr_reduction}. By Theorems~\ref{dr_reduction} and~\ref{wrochna} we can conclude that $(K_3,K_3)$ finitely reduces to $(\hh_2,\hh_c)$ which means that $K_{K_3}$, and thus also $\ufl$ holds.
\end{proof}

\begin{theorem}\label{colours}
	Any of the following statements are equivalent to $\ufl$ over $\zf$.
\begin{enumerate}[(1)]
\item\label{it:maink} $K_{(K_k,K_c)}$ for any $k,c\in \omega$ with $3\leq k\leq c\leq 2k-1$.
\item\label{it:mainh} $K_{(\hh_2,\hh_c)}$ for any $c\in \omega, c\geq 2$.
\end{enumerate}
\end{theorem}

\begin{proof}
	Follows from Theorems~\ref{thm:main} and~\ref{graph_promise}.
\end{proof}

	We finally mention that by a relatively easy set theoretical argument one can also remove the uniform upper bounds in item~\ref{it:mainh} of Theorem~\ref{colours}.

	

\begin{definition}	
	For a $k\in \omega$ we say that a 3-uniform hypergraph $\fg$ is \emph{$k$-colourable} if it has a homomorphism to $\hh_k$, and \emph{finitely $k$-colourable} if every finite subhypergraph of $\fg$ has a homomorphism to $\hh_k$.
\end{definition}
	
\begin{theorem}[$\zf$]\label{colours2}
	Let us assume that every finitely 2-colourable 3-uniform hypergraph can be coloured by finitely many colours. Then $\ufl$ holds.
\end{theorem}

\begin{proof}
	We show that the assumption of the lemma implies $K_{(\hh_2,\hh_c)}$ for some $c\geq 2$. Suppose that this is not the case. Then for all $c\geq 2$ there exists some 3-uniform hypergraph $\fg_c$ which is finitely $2$-colourable but not $c$-colourable. Let $\alpha_c$ be the minimal rank (in the cumulative hierarchy) of such a hypergraph, and let $\alpha\coloneqq \sup_c\alpha_c$. Then $S\coloneqq V(\alpha+1)$ is a set containing some hypergraph $\fg_c$ as above for all $c\in \omega$. Let $\fg$ be the disjoint union of all finitely 2-colourable hypergraphs in $S$. Then by our construction $\fg$ cannot be coloured by finitely many colours, a contradiction.
\end{proof}

\begin{remark}
	Note that in our proof above we used the Axiom of Foundation, but this can be avoided by the following argument. Let us assume that $\ufl$ is false, or equivalently $K_{K_3}$ is false, and let $G$ be a graph which is not 3-colourable but not finitely 3-colourable. Let $S$ be any set such that $G\in S$, and $S$ is closed under the operators $\bigcup$ and $\power$. The existence of such a set follows from basic set theoretical arguments. Then one can easily check that in all finite reductions provided by the our proof of Theorem~\ref{dr_reduction} in Subsection~\ref{sect:main} we never leave the set $S$ (this relies on the specific details of our proof). From this we can conclude that for all $c\in \omega$ there exists some hypergraph $\fg_c\in S$ which is finitely 2-colourable but not $c$-colourable. Then by the same construction as in the proof of Theorem~\ref{colours} we can find a hypergraph $\fg$ (not necessarily in $S$) which is finitely 2-colourable but cannot be coloured by finitely many colours.
\end{remark}


\subsection{The proof of Theorem~\ref{dr_reduction}}\label{sect:main}

	We first recall some technical notions and results from~\cite{barto2021algebraic} which we will need in our reduction.
	
\begin{definition}
	For a set of variables $V$ and some set $A$, a \emph{partial assignment system (PAS) of arity $k$}, or \emph{$k$-PAS}, over $A$, is a map $\mathcal{I}$ defined on ${V \choose k}$ such that for each $U\in {V\choose k}$ we have $\mathcal{I}(\mathcal{U})\subseteq \prescript{U\!}{}{A}$. The value of a $k$-PAS $\mathcal{I}$, denoted by $\val(\mathcal{I})$, is the maximal size of $\mathcal{I}(\mathcal{U}): U\in {V \choose k}$.
	
	A map $f: V\rightarrow A$ is an \emph{$m$-solution} of a $k$-PAS $\mathcal{I}$, if every $U\in {V\choose m}$ can be extended to some $W\in {V\choose k}$ such that $f|_U\in \mathcal{I}(W)|_U$.
\end{definition}

\begin{definition}
	Let $(\mathcal{I}_0,\dots,\mathcal{I}_r)$ be a sequence of PASes over $A$ such that the arity of $\mathcal{I}_i$ is $k_i$. We call such a sequence \emph{consistent} if
\begin{itemize}
\item $k_0\geq k_1\geq \dots \geq k_r$, and
\item for all $V\supseteq U_0\supseteq U_1\supseteq \dots \supseteq U_r$ with $|U_i|=k_i$ there exist $0\leq i<j\leq r$ such that $\mathcal{I}_j\cap \mathcal{I}_i|_{U_j}\neq \emptyset$.
\end{itemize}
\end{definition}

	Next we state the main result of~\cite{barto2021algebraic} (Theorem 2.1) which will be used in our reduction as well.
	
\begin{theorem}[$\zfc$]\label{pas}
	For all $n,m,r,d\in \omega$ there exist $k_0,\dots,k_r\in \omega$ such that for any consistent sequence $(\mathcal{I}_0,\dots,\mathcal{I}_r)$ of PASes of arities $k_0,\dots,k_r$ over $n$ and with values at most $d$ there exists an $i\in \{0,1,\dots,r\}$ such that $\mathcal{I}_i$ has an $m$-solution. 
\end{theorem}

	Note that we stated Theorem~\ref{pas} as a $\zfc$ theorem. In the context of~\cite{barto2021algebraic} this distinction is irrelevant since only finite instances of PASes are considered, and clearly in this case the proof of Theorem~\ref{pas} goes through in $\zf$. Nevertheless, the proof itself also works for PASes of any size, however, in this case the reduction step in the proof seemingly needs to use some choice at some point. We show, however, that the original proof goes through in $\zf+\acf$ with a slight modification.

\begin{theorem}\label{pas2}
	Theorem~\ref{pas} holds in $\zf+\acf$.
\end{theorem}

	The following two notions are essentially taken from~\cite{barto2021algebraic}.
	
\begin{definition}
	Let $\mathcal{I}$ be a $k$-PAS over $n$ with a variable set $V$, let $f$ be a partial map $V\rightarrow n$ with $|\dom(f)|\leq k$. Then for some $\ell\leq k$ we say that $f$ is
\begin{itemize}
\item \emph{$\ell$-obstacle (with respect to $\mathcal{I}$)} if for all $W\in {V\choose \ell}$ there exist some $U\in {V\choose k}$ and $g\in \mathcal{I}(U)$ such that $\dom(f)\cup W\subseteq U$ and $f\subseteq g$;
\item \emph{$\ell$-admissible (with respect to $\mathcal{I}$)} if there exists some $W\in {V\choose \ell}$ such that for all $U\in {V\choose k}$ with $\dom(f)\cup W\subseteq U$ there exists some $g\in \mathcal{I}(U)$ such that $f\subseteq g$.
\end{itemize}
\end{definition}

\begin{remark}
	Using the terminology of~\cite{barto2021algebraic} $f$ being an $\ell$-obstacle means that $(\dom(f),f)$ has the $\ell$-property $P$ and $f$ being $\ell$-admissible means that $(\dom(f),f)$ does not have the $\ell$-property $Q$.
\end{remark}

	The following lemma is a reformulation of Proposition A.1 in~\cite{barto2021algebraic} (the proof for this is not using any choice axioms).
	
\begin{lemma}[$\zf$]\label{p}
	Let $\mathcal{I}$ be a $k$-PAS over $n$ with a variable set $V$, and let $X$ be an at most $k$-element subset of $V$. If $k\geq n^{|X|}\ell+|X|$ then there exists some $\ell$-obstacle with $\dom(f)=X$.
\end{lemma}

	We adapt the following notion from the proof of Theorem~\ref{pas} in~\cite{barto2021algebraic}.

\begin{definition}
	We say that an $\ell$-PAS $\mathcal{J}$ is a \emph{refinement} of a $k$-PAS $\mathcal{I}$, or $\mathcal{J}$ \emph{refines} $\mathcal{I}$, if $\ell\leq k$ and for all $U\in {V\choose \ell}$ there exists an $\tilde{U}\in {V\choose k}$ such that $\mathcal{J}(U)=\mathcal{I}(\tilde{U})|_U$.
\end{definition}

	It is clear from the definition that if $\mathcal{J}$ refines $\mathcal{I}$ then $\val(\mathcal{J})\leq \val(\mathcal{I})$.
	
	Note that, as opposed to the analogous notion introduced in~\cite{barto2021algebraic} we do not require that there exists an extension function $\ex: {V\choose \ell}\rightarrow {V\choose k}, U\mapsto \tilde{U}$ witnessing the refinement in the definition above since the existence of such a function cannot necessarily be established without any choice axioms. However, as we will see, the proof itself never requires the existence of the function $\ex$ itself, as we can define refinements by picking one of the possible restrictions to the smaller subsets. Since for a given $\ell$-element subset there are only finitely many restrictions such choices are possible to make under the assumption of $\acf$.
	
	Our next lemma will be used in the case where we can find an admissible function everywhere.

\begin{lemma}[$\zf+\acf$,~\cite{barto2021algebraic}, Proposition A.2]\label{q}
	Let $\mathcal{I}$ be a $k$-PAS as above. Let us assume that $k\geq k'+{k''\choose k'}\ell$ and for all $X\in {V\choose k'}$ there exists an $\ell$-admissible function with domain $X$. Then there exists a $k'$-PAS $\mathcal{I}'$ of value 1 and a $k''$-PAS $\mathcal{I}''$, a refinement of $\mathcal{I}$, so that $(\mathcal{I}'',\mathcal{I}')$ is consistent.
\end{lemma}

\begin{proof}
	Define $\mathcal{I}'(X)=\{f\}$ where $f$ is an $\ell$-admissible with $\dom(f)=X$. This is possible assuming $\acf$ since we always have finitely many choices for the function $f$ on each $X\in {V\choose k'}$. We define $\mathcal{I}''$ as follows. Let $Y\in {V\choose k''}$ be arbitrary. Let $\nu: {Y\choose k'}\rightarrow {V\choose \ell}$ be a function such that $W=\nu(X)$ witnesses the $\ell$-admissibility of $\mathcal{I}'(X)$, and let us define $\ex(Y)$ to be any $k$-element subset of $V$ containing $Y\cup \bigcup_{X\in {Y\choose k'}}\nu(X)$. This is possible by our assumption on the values of $k,k',k''$ and $\ell$. Then we want to define $\mathcal{I}''(Y)\coloneqq \mathcal{I}(\ex(Y))|_X$. Again, without any choice axiom this is not necessarily possible. Note, however, that we always have finitely many choices for $\mathcal{I}''(Y)$ for any given $Y\in {V\choose k''}$, so under the assumption of $\acf$ the values of $\mathcal{I}''(Y)$ can be picked at the same time. The consistency of $(\mathcal{I}'',\mathcal{I}')$ is then clear from the definition.
\end{proof}

	Now we are ready to prove Theorem~\ref{pas2}. 
	In order to do this, we show the following more general lemma. We obtain Theorem~\ref{pas2} in the case when $d_0=d_1=\dots=d_r=d$.
	
\begin{lemma}[$\zf+\acf$]\label{pas3}
	For all $n,m,r,d_0,\dots,d_r\in \omega$ there exists a sequence $(k_0,\dots,k_r)\in \prescript{r+1\!}{}{\omega}$ such that for any consistent sequence $(\mathcal{I}_0,\dots,\mathcal{I}_r)$ of PASes over $n$ of arity $(k_0,\dots,k_r)$ with $\val(\mathcal{I}_j)\leq d_j$ there exists an $i\in \{0,1,\dots,r\}$ such that $\mathcal{I}_i$ has an $m$-solution. 
\end{lemma}

\begin{remark}
	Note that Theorem~\ref{pas2} applied with $d=\max(d_0,\dots,d_r)$ implies Lemma~\ref{pas3}, so Lemma~\ref{pas3} and Theorem~\ref{pas2} are in fact equivalent.
\end{remark}

\begin{proof}
	Our proof is again a reformulation of the proof of Theorem~\ref{pas} presented in~\cite{barto2021algebraic}.
	
	
	We will show that for all sequences $(d_0,\dots,d_r)$ if Lemma~\ref{pas3} holds for the sequences $(d_1,1),\dots,(d_r,1)$ and if $d_0\geq 1$ then it holds for the sequence $(d_0-1,\dots,d_r)$ then it holds for the sequence $(d_0,\dots,d_r)$ as well ($*$). This proves Lemma~\ref{pas3} by running an induction on $s\coloneqq r+d_0+2\sum_{i=1}^rd_i$.
	

	If $(\mathcal{I}_0,\dots,\mathcal{I}_r)$ is a consistent sequence then at least two different $\mathcal{I}_i$ must have a value at least 1. Thus, we can assume without loss of generality that at least 2 of the values $d_i$ are at least 1. Under this assumption the minimal value of $s$ is 4 which is only attained if $r=d_0=d_1=1$. In this case, we can put $k_0=m$ and $k_1=1$ since in this case the map $f\in \prescript{V\!}{}{n}$ defined by the union of the single functions in $\mathcal{I}_1$ provides an $m$-solution. From this point on we will assume that this is not the case, i.e., either $r\geq 2$ or $d_i\geq 2$ for some $i$. In this case $(*)$ indeed works as an induction step since the validity of the lemma is only assumed for tuples with smaller $s$ values. We will also assume that $d_0\geq 1$, otherwise we can remove the PAS $\mathcal{I}_0=\emptyset$.
	
	For $i\in \{1,\dots,r\}$ let $(k_i',k_i'')$ be a pair as in the conclusion of Lemma~\ref{pas2} applied to $(d_i,1)$, and let $(p_0,\dots,p_r)$ be a sequence that works for the tuple $(d_0-1,d_1,\dots,d_r)$.
	
	Next, we define the sequences $(k_0,\dots,k_r)$ and $(\ell_0,\dots,\ell_r)$ simultaneously and going backwards as follows. We write $k_{r+1}=p_{r+1}\coloneqq 0$.
	
\begin{itemize}
\item For $i$ equals $r,r-1,\dots,1,0$ we define $\ell_i\coloneqq p_i+{p_i\choose p_{i+1}}(k_{i+1}-p_{i+1})$ and
\item for $i$ equals $r,r-1,\dots,1$ we define $k_i\coloneqq k_i''+{k_i''\choose k_i'}\ell_i$.
\item Finally, we put $k_0\coloneqq \sum_{j=1}^rk_j'+n^{\sum_{j=1}^rk_j'}\ell_0$.
\end{itemize}

	We show that the sequence $(k_0,\dots,k_r)$ suffices.
	
	\emph{Case 1. There exists an $i\in \{1,\dots,r\}$ such that for all $X\in {V\choose k_i'}$ there exists some $f: X\rightarrow n$ which is $\ell_i$-admissible with respect to $\mathcal{I}_i$.} 

	Then by Lemma~\ref{q} there exist a $k'$-PAS $\mathcal{I}'$ of value 1, and a $k''$-PAS $\mathcal{I}''$ which is a refinement of $\mathcal{I}_i$ such that $(\mathcal{I}'',\mathcal{I}')$ is consistent. Since $\mathcal{I}''$ refines $\mathcal{I}_i$ it follows that $\val(\mathcal{I}'')\leq d_i$, therefore we can apply the induction hypothesis to the pair $(\mathcal{I}'',\mathcal{I}')$. We obtain that either $\mathcal{I}''$ or $\mathcal{I}'$ has an $m$-solution. If $\mathcal{I}'$ has an $m$-solution then it must be $h\coloneqq \bigcup\{f: \{f\}\in \mathcal{I'}\}$ and by consistency $h$ also must be an $m$-solution to $\mathcal{I}''$. So $\mathcal{I}''$ has an $m$-solution in either case, and since it is a refinement of $\mathcal{I}_i$, it follows that the same $m$-solution works for $\mathcal{I}_i$.
	
	\emph{Case 2. There exist $X_i\in {V\choose k_i'}: i\in \{1,\dots,r\}$ such that for all $i$ there is no $f:X_i\rightarrow n$ which is $\ell_i$-admissible with respect to $\mathcal{I}_i$.} Let $X\coloneqq \bigcup_{i=1}^rX_i$. Then by Lemma~\ref{p} we can find an $f:X\rightarrow n$ which is an $\ell_0$-obstacle for $\mathcal{I}_0$. Then we define $f_i\coloneqq f|_{X_i}$. By our assumption for all $i\in \{1,\dots,r\}$ the map $f_i$ is not $\ell_i$-admissible with respect to $\mathcal{I}_i$. We write $W\sqsubset_i U$ if $X_i\cup W\subseteq U$ and $g|_{X_i}\neq f_i$ for all $g\in \mathcal{I}_i(U)$. Since $f_i$ is not $\ell_i$-admissible this implies that for all $W$ the set $\{U: W\sqsubset_i U\}$ is nonempty.
	
	Now we define the sequences $(\mathcal{J}_1,\dots,\mathcal{J}_r)$ and $(\ex_0,\dots,\ex_r)$ where $\mathcal{J}_i$ is $p_i$-PAS which is a refinement of $\mathcal{I}_i$, and $\ex_i$ is a function from ${V\choose p_i}$ to $\power({V \choose k_i})$ that will consist of some specific witnesses of $\mathcal{J}_i$ being a refinement of $\mathcal{I}_i$. We point out that this is where our proof slightly diverges from the one presented in~\cite{barto2021algebraic} where first $\ex_i$ is defined as a function from ${V\choose p_i}$ to ${V\choose k_i}$ and $\mathcal{J}_i$ is defined as the refinement of $\mathcal{I}_i$ via $\ex_i$. However, in order to make this original argument work (with infinitely many variables) we would need to invoke $\ac$. Our argument avoids this at the cost of some additional technical difficulties (and $\acf$ is still needed).

	We first define the pairs $(\mathcal{J}_i,\ex_i)$ for $i\geq 1$ recursively and going backwards. We put $\ex_{r+1}\coloneqq \emptyset$.
	
	Assuming that $\ex_{i+1}$ is already defined, we define $\mathcal{J}_i: i\geq 1$ and $\ex_i$ as follows. Let $Y\in {V\choose p_i}$ be arbitrary. Then we pick some function $\alpha_i: {Y\choose p_{i+1}}\rightarrow {V\choose k_{i+1}}$ satisfying $\alpha_i(Z)\in \ex_{i+1}(Z)$ for all $Z\in {Y\choose p_{i+1}}$, and we define $$W_i\coloneqq Y\cup \bigcup_{Z\in {Y\choose p_{i+1}}}\alpha_i(Z).$$ Note that in this case $|W_i|\leq \ell_i$, so we can find some $U_i\in {V\choose k_i}$ such that $W_i\sqsubset_i U_i$. Now we want to define $\mathcal{J}_i(Y)$ to be $\mathcal{I}_i(U_i)|_Y$. This is possible to do simultaneously using $\acf$ since all $\mathcal{J}_i(Y)$ have finitely many possible options. More precisely, for a $Y\in {V\choose p_i}$ we say that a set $S$ of functions in $\prescript{Y\!}{}{n}$ is \emph{good} if $S=\mathcal{I}_i(U_i)|_Y$ for some $\alpha_i,W_i,U_i$ are as described above. Then for all such $Y$ we have finitely many good choices, and therefore using $\acf$ we can define $\mathcal{J}_i$ to be a choice function which assigns to each $Y\in {V\choose p_i}$ a good set for $Y$. Finally, we put a set $X\in {V\choose k_i}$ in $\ex_i(Y)$ if and only if there exist $\alpha_i,W_i$ and $U_i$ above so that $\mathcal{I}_i(U_i)|_Y=\mathcal{J}_i(Y)$ and $X=U_i$.
	
	Next, we define the $p_0$-PAS $\mathcal{J}_0$. Let $Y\in {V\choose p_0}$ be arbitrary. We define the map $\alpha_0$ and the set $W_0$ the same way as above, note that in this case $|W_0|=\ell_0$. Since $f$ is an $\ell_0$-obstacle with respect to $\mathcal{I}_0$ we can pick some $U\in {V\choose k_0}$ such that $X\cup W_0\subseteq U$ and $f=g|_{W_0}$ for some $g\in \mathcal{I}_0$. Then we want to define $\mathcal{J}_0(Y)$ to be $\{g|_Y: g\in\mathcal{I}_0(U): g|_X\neq f\}$. Again, by the same argument as above, this is possible to do simultaneously using $\acf$. Note that the definition of $\mathcal{J}_0$ immediately implies that it is a refinement of $\mathcal{I}_0$ and $\val(\mathcal{J}_0)\leq d_0-1$ (since at least one possible value is always removed from $\mathcal{I}_0$). Finally, we put $U$ in $\ex_0(Y)$ if and only if there exist some $\alpha_0$, $W_0$ and $g$ as above with $\{g|_Y: g\in\mathcal{I}_0(U): g|_X\neq f\}=\mathcal{J}_0(Y)$.
	
	We show that the sequence $(\mathcal{J}_0,\dots,\mathcal{J}_r)$ is consistent. Let $Y_0\supseteq \dots \supseteq Y_r$ be a sequence with $|Y_i|=p_i$. We have to show that $\mathcal{J}_i(Y_i)|_{Y_j}\cap \mathcal{J}_j(Y_j)\neq \emptyset$ for some $i<j$. Let $U_0$ be as in the definition of $\mathcal{J}_0(Y_0)$. Then we define a sequence $U_0\supseteq U_1\supseteq \dots \supseteq U_r$ of subsets of $V$ such that $|U_i|=k_i$ and $U_i\in \ex_i(Y_i)$ for all $i\geq 1$. Let us assume that $U_0,\dots,U_i$ are already defined. Since $U_i\in \ex_i(Y_i)$ we can find some maps $\alpha_i,W_i$ witnessing this in the definition of $\mathcal{J}_i$. Let $U_{i+1}\coloneqq \alpha_i(Y_{i+1})$. Then by definition $U_{i+1}\in \ex_{i+1}(Y_{i+1})$ and $U_{i+1}\subseteq W_i\subseteq U_i$. Let us also observe that since $U_i\in \ex_i(Y_i)$ we have $\mathcal{J}_i(Y_i)=\mathcal{I}_i(U_i)|_{Y_i}$. Since $U_0\supseteq U_1\supseteq \dots \supseteq U_r$ and by our assumption the original sequence $(\mathcal{I}_0,\mathcal{I}_r)$ is consistent, we obtain that there exist $1\leq i<j\leq r$ such that $\mathcal{I}_i(U_i)|_{U_j}\cap \mathcal{I}(U_j)\neq \emptyset$. Let $g$ be a function in this intersection. We show that $g|_{Y_j}\in \mathcal{J}_i(Y_i)|_{Y_j}\cap \mathcal{I}(Y_j)$. If $1\leq i$ this is immediate, so let us assume that $i=0$. Let $g'\in \mathcal{I}_0(U_0)$ such that $g'|_{U_j}=g$. By definition we know that $W_j\sqsubset U_j$ for some $W_j\in {V\choose \ell_j}$, in particular $g|_{X_j}\neq f_j$. This implies $g'|_X\neq f$, and thus $g'|_{Y_0}\in \mathcal{J}_0(Y_0)$. Considering that $(g'|_{Y_0})|_{Y_j}=g'|_{Y_j}=g|_{Y_j}$ we have $g|_{Y_j}\in \mathcal{J}_i(Y_i)|_{Y_j}\cap \mathcal{I}(Y_j)$.

	We have shown that $(\mathcal{J}_0,\dots,\mathcal{J}_r)$ is consistent which, by the induction hypothesis, implies that $\mathcal{J}_i$ has an $m$-solution for some $i$. By our construction, this also provides an $m$-solution to $(\mathcal{I}_0,\dots,\mathcal{I}_r)$ which finishes the induction step.
\end{proof}


	Now we turn into proving Theorem~\ref{dr_reduction}. In our proof we follow the construction described in the proof of Theorem 5.1 in~\cite{barto2021algebraic}. While the original construction was used to give a polynomial time reduction between the corresponding PCSPs, we can also use it to give a finite reduction. 

	In our reduction we need an intermediate promise template which requires the notion of free structures over minions.
	
\begin{definition}
	Let $\fs$ be a finite relational structure with $\dom(\fs)=n$, and let $\mathcal{M}$ be a minion. Then the \emph{free structure of $\mathcal{M}$ generated by $\fs$}, denoted by $\free_{\mathcal{M}}(\fs)$, is the structure similar to $\fs$ whose domain set is the set of $n$-ary functions in $\mathcal{M}$ and for a $k$-ary relational symbol $R$ in the signature of $\fs$ we define $R^{\free_{\mathcal{M}}(\fs)}$ to be the set of those $k$-tuples $(f_0,\dots,f_{k-1})$ from $\mathcal{M}^{(n)}$ such that there exists an $m$-ary function $g\in \mathcal{M}$ satisfying
	$$f_i(x_0,\dots,x_{n-1})=g(x_{r_0(i)},\dots,x_{r_{m-1}(i)})$$ where $m=|R^\fs|$ and $r_0,\dots,r_{m-1}$ are tuples enumerating the relation $R^{\fs}$.
\end{definition}

	We know that for any $\fs$ and $\mathcal{M}$ as above there exists a minion homomorphism from $\mathcal{M}$ to $\pol(\fs,\free_{\mathcal{M}}(\fs))$, see for instance the discussion after Definition 4.1 in~\cite{barto2021algebraic}. Thus, by Theorem~\ref{minions} we know that if $\mathcal{M}=\pol(\fa,\fb)$ and $\fa$ and $\fb$ are finite then $(\fa,\fb)$ pp-constructs $(\fs,\free_{\mathcal{M}}(\fs))$. Combining this with Theorem~\ref{construct_reduce}, we can conclude the following.

\begin{lemma}[$\zf$]\label{free_reduction}
	Let $\fa,\fb,\fs$ be finite structures with $\dom(\fs)\in \omega$, and let $\mathcal{M}\coloneqq \pol(\fa,\fb)$. Then $(\fs,\free_{\mathcal{M}}(\fs))$ finitely reduces to $(\fa,\fb)$.
\end{lemma}
	
	Now we are ready to prove Theorem~\ref{dr_reduction}.

\begin{proof}
	Let $\mathcal{M}\coloneqq \pol(\fa,\fb)$. Then we construct an appropriate structure $\fs$ for which we can show that $(\fc,\fd)$ finitely reduces to $(\fs,\free_{\mathcal{M}}(\fs))$. By Lemma~\ref{free_reduction} this implies Theorem~\ref{dr_reduction}.
	
	By our assumption there exists a $(d,r)$-minion homomorphism $\xi$ from $\mathcal{M}$ to $\pol(\fc,\fd)$ for some $d,r\in \omega$. Let $k_0,\dots,k_r$ be as in the conclusion of Theorem~\ref{pas} applied in the case when $n=|\dom(\fd)|$ and $m$ is the maximal arity of the relations of $\fc$ (and thus also of $\fd)$. By Theorem~\ref{pas2} the existence of such numbers is provable in $\zf+\acf$. 
	
	We define the structure $\fs$ as follows. The domain set of $\fs$ is $S\coloneqq |C|^{k_0}$ and its relations are all partial functions from $S$ to $S$ (each of them represented by itself). Then we define an operation $\Gamma$ which maps instances of $(\fc,\fd)$ to $(\fs,\free_{\mathcal{M}}(\fs))$. Let $\ii$ be a structure similar to $\fc$. The domain set of $\Gamma(\ii)$ is defined as $\mathcal{X}\coloneqq \bigcup_{i=0}^r{V\choose k_i}$. For an $U\in \mathcal{X}$ we write $D_U$ for the set of homomorphisms from $\ii|_U$ to $\fc$. Then for each $U\in \mathcal{X}$ we pick some bijection $\sigma_U\colon D_U\rightarrow |D_U|$, note that $|D_U|\leq S$. This is possible to do simultaneously by $\acf$.  Then for all pairs $U,W\in \mathcal{X}$ with $W\subseteq U$ we put the pair $(U,W)$ in the relation $\{(\sigma_U(f),\sigma_W(f|_W)): f\in D_U\}$ in $\Gamma(\ii)$.
	
	We claim that $\Gamma$ gives a finite reduction from $(\fc,\fd)$ to $(\fs,\free_{\mathcal{M}}(\fs))$.
	
	Let us first observe that if $h: \ii\rightarrow\fc$ is a homomorphism then then $U\mapsto \sigma_U(h|_U)$ is a homomorphism from $\Gamma(\ii)$ to $\fs$. Next we check $\Gamma$ that satisfies item~\ref{it:def3} in Definition~\ref{def:fin} by using Lemma~\ref{lem:fin}. If $\fh$ is a finite substructure of $\Gamma(\ii)$ then we define $F\coloneqq \bigcup_{U\in H}U$, and we write $\ff$ for the substructure of $\ii$ induced on $F$. Then $\ff$ is finite, and $\fh$ is a substructure of $\Gamma(\ff)$. Therefore, the hypotheses of Lemma~\ref{lem:fin} hold, and thus $\Gamma$ satisfies item ~\ref{it:def3}.
	
	It remains to show that item~\ref{it:def2} in Definition~\ref{def:fin} holds for $\Gamma$, i.e, if $s\colon \Gamma(\ii)\rightarrow \free_{\mathcal{M}}(\fs)$ then $\ii\rightarrow\fd$. This essentially follows from the same argument as the one used for proving soundness in the proof of Theorem B.2 in~\cite{barto2021algebraic}. For the sake of completeness, we present this argument here as well. 

	We use the notation $\iota_m$ for the map $\id(m)\cup\,((S\setminus m)\times \{0\})$, i.e., the identity map on $m$ extended by the constant 0 map on $S\setminus m$. Let $U\in {V\choose k_i}$. We define the $|D_U|$-ary function $t(U)$ in $\mathcal{M}$ by $t(U)\coloneqq (s(U))_{\iota_{|D_U|}}$, i.e., $$t(U)(x_0,\dots,x_{|D_U|-1})=s(U)(x_0,\dots,x_{|D_U|-1},x_0,\dots,x_0).$$ We claim that if $W\subseteq U$ then $t(U)\xrightarrow{\pi_{U,W}}t(W)$ where $\pi_{U,W}: |D_U|\rightarrow |D_W|, m\mapsto \sigma_W((\sigma_U^{-1})(m)|_W)$. By definition $(U,W)$ is contained in the relation $\{(\sigma_U(f),\sigma_W(f|_W)): f\in D_U\}$ and $s$ is a homomorphism, it follows that there exists some $g\in \mathcal{M}$ of arity $|D_U|$ such that

\begin{itemize}
\item $s(U)(x_0,\dots,x_{S-1})=g(x_0,\dots,x_{|D_U|-1})$ and
\item $s(W)(x_0,\dots,x_{S-1})=g(x_{\pi_{U,W}(0)},\dots,x_{\pi_{U,W}(|D_U|-1)})$.
\end{itemize}

	This also implies that $s(U)$ only depends on the first $|D_U|$ and $s(W)$ only depends on the first $|D_W|$ coordinates. In particular, $g=t(U)$ and $g_{\pi_{U,W}}=t(W)$ which finishes the proof of our claim.
	
	We construct a sequence $(\mathcal{I}_0,\dots,\mathcal{I}_r)$ of PASes of arities $k_0,\dots,k_r$ on $V=\dom(\ii)$ as follows. Let $U\in {V\choose k_i}$. Then for a $g\in \xi(t(U))$ we define $$Z_U(g): U\rightarrow D, u\mapsto g(\sigma^{-1}_U(0)(u),\dots,\sigma^{-1}_U(|D_U|-1)(u)),$$ and we put $$\mathcal{I}_i(U)\coloneqq \{Z_U(g): g\in \xi(t(U))\}.$$ It is clear from the definition that all $\mathcal{I}_i$ has value at most $d_i$. Moreover, since $\xi(t(U))\in \pol(\fc,\fd)$ it follows that all elements of $\mathcal{I}_i$ are partial homomorphisms to $\fd$. We now show that the sequence $(\mathcal{I}_0,\dots,\mathcal{I}_r)$ is consistent. Let $V\supseteq U_0\supseteq \dots \supseteq U_r$ with $|U_i|=k_i$. Then, as we have seen, $$t(U_0)\xrightarrow{\pi_{U_0,U_1}}t(U_1)\xrightarrow{\pi_{U_1,U_2}}\dots \xrightarrow{\pi_{U_{r-1},U_r}}t(U_r)$$ is a sequence of minors in $\mathcal{M}$. Since $\xi$ is a $(d,r)$-minion homomorphism, it follows that there exist $i<j,g\in \xi(t(U_i)),h\in \xi(h(U_j))$ such that $g\xrightarrow{\pi_{U_i,U_j}} h$. Then for all $u\in U_j$ we have
\begin{align*}
Z_{U_j}(h)(u)=&h(\sigma^{-1}(0)(u),\dots,\sigma^{-1}(|D_W|-1)(u))\\=&g_{\pi(U_i,U_j)}(\sigma^{-1}(0)(u),\dots,\sigma^{-1}(|D_W|-1)(u))\\=&g(\sigma^{-1}(0)(u),\dots,\sigma^{-1}(|D_{U_i}|-1)(u))=Z_{U_i}(g)(u).
\end{align*}
	
	This means that $Z_{U_i}(g)|_{U_j}=Z_{U_j}(h)$, and thus $\mathcal{I}_i|_{U_j}\cap \mathcal{I}_j\neq \emptyset$ which shows that $(\mathcal{I}_0,\dots,\mathcal{I}_r)$ is consistent.
	
	By the choice of $k_0,\dots,k_r$ we can conclude that $\mathcal{I}_i$ has an $m$-solution $h$ for some $i$. Since $\mathcal{I}_i$ consist of partial homomorphisms to $\fd$ it follows that the restriction of $h$ to every at most $m$-element substructure of $\ii$ is a homomorphisms to $\fd$. Since all relations of $\fd$ has arity at most $m$ this means that $h$ is in fact a homomorphism from $\ii$ to $\fd$.
\end{proof}

\section{Conclusion and open problems}\label{sect:open}

	As we have seen, a lot of analogies can be drawn between reduction between promise CSPs and the reduction between the corresponding compactness principles, and we have shown that for several promise templates which are known to have hard CSPs, the corresponding compactness principles are as strong as possible, i.e., they are equivalent to the ultrafilter lemma. It would be interesting to see how far these analogies can be pushed, for example one can ask for any given promise template $(\fa,\fb)$ for which $\pcsp(\fa,\fb)$ is known or conjectured to be \npc, whether $K_{(\fa,\fb)}$ is still equivalent to the ultrafilter lemma. 
	
\begin{question}
	Is $K_{(K_3,K_6)}$ equivalent to $\ufl$ over $\zf$?
\end{question}

	More generally, we can make the following conjecture.
	
\begin{conjecture}\label{conj:promise}
	For all $3\leq c\in \omega$ the compactness principle $K_{(K_3,K_c)}$ is equivalent to $\ufl$ over $\zf$.
\end{conjecture}

	
	It could also be worthwhile to look at some of the other methods for proving hardness of PCSPs in some of the papers mentioned in the introduction (\cite{krokhin2023topology,avvakumov2025hardness,austrin20172+varepsilon,filakovsky2023hardness}), to see if they can be adapted in the context of compactness principles.
	
	We can also turn around the questions above, and ask whether there exists any promise template $(\fa,\fb)$ which has a tractable PCSP but $K_{(\fa,\fb)}$ still implies $\ufl$. Note that by the results of~\cite{katay2023csp} we know that this is impossible in the case when $\fa=\fb$. A promise template $(\fa,\fb)$ is called \emph{finitely tractable} if there exists some finite, not omniexpressive structure $\fc$ such that $(\fa,\fb)$ is a homomorphic relaxation of $(\fc,\fc)$. Note that if $(\fa,\fb)$ is finitely tractable then it follows from the finite-domain CSP dichotomy theorem~\cite{BulatovFVConjecture,zhuk2020proof} that $\pcsp(\fa,\fb)$ is tractable. Moreover, in this case $K_{(\fa,\fb)}$ is true in a model of $\zf$ as in Theorem~\ref{thm:ktv}, in particular it does not imply $\ufl$. The prototypical example for a promise template which is not finitely tractable, but has a tractable PCSP is $(T,H_2)$ where $T$ is the template structure for 1-in-3-SAT, i.e., $\dom(T)=2$, and its only (ternary) relation is $\{(0,0,1),(0,1,0),(1,0,0)\}$\cite{brakensiek2018promise,barto2021algebraic}.
	
\begin{question}
	Is $K_{(T,H_2)}$ equivalent to $\ufl$ over $\zf$?
\end{question}
	
	A slightly unsatisfactory aspect of the results of Section~\ref{sect:weak} is that in order to make the general reductions work we need to invoke $\acf$. This motivates the following question.
	
\begin{question}
	Is Theorem~\ref{dr_reduction} or Theorem~\ref{pas} provable in $\zf$?
\end{question}

	Finally, one can ask whether the converse of any of the main results of Section~\ref{sect:ac_fin} holds.
	
\begin{question}
	Does there exist a promise template $(\fa,\fb)$ which has a cyclic polymorphism of arity $p$ for some prime $p$ and $K_{(\fa,\fb)}$ implies $\kw(p)$ over $\zf$.
\end{question}

\begin{question}\label{no_cyc}
	Does there exist a promise template $(\fa,\fb)$ which has a cyclic polymorphism and $K_{(\fa,\fb)}$ implies $\acf$ over $\zf$?
\end{question}


\bibliographystyle{alpha}
\bibliography{local.bib}

\end{document}